\documentclass[a4paper, 12pt,oneside,reqno]{amsart}
\usepackage[a4paper]{geometry}
\usepackage[english]{babel}
\usepackage[T1]{fontenc}
\usepackage{amsfonts}
\usepackage{mathrsfs}
\usepackage{tikz}
\usetikzlibrary{arrows,shapes,snakes,automata,backgrounds,petri,through,positioning}
\usetikzlibrary{intersections}

\usepackage[matrix,arrow]{xy}
\usepackage{amssymb,amscd,amsthm,amsmath}
\usepackage[a4paper]{geometry}
\usepackage{amsmath}
\usepackage{amssymb}
\usepackage{amsthm}
\usepackage[colorlinks=true, allcolors=blue]{hyperref}

\newtheorem{theorem}{Theorem}[section]
\newtheorem{lemma}[theorem]{Lemma}

\newtheorem{proposition}[theorem]{Proposition}
\newtheorem{corollary}[theorem]{Corollary}
\theoremstyle{definition}
\newtheorem{definition}[theorem]{Definition}
\newtheorem{example}[theorem]{Example}
\newtheorem{remark}[theorem]{Remark}

\newcommand{\CC}{\mathbb{C}}
\newcommand{\ZZ}{\mathbb{Z}}
\newcommand{\QQ}{\mathbb{Q}}
\newcommand{\TT}{\mathbb{T}}
\newcommand{\HH}{\mathbb{H}}
\newcommand{\KK}{\mathbb{K}}
\newcommand{\Aut}{\mathrm{Aut}}
\newcommand{\Ker}{\mathrm{Ker}}
\newcommand{\Lie}{\mathrm{Lie}}

\renewcommand{\phi}{\varphi}

\title{Isotropy subgroups of homogeneous locally nilpotent derivations}
\author{Dmitriy Chunaev}
\address{Lomonosov Moscow State University
	National Research University Higher School of Economics, Moscow}
\email{dchunaev@hse.ru}
\author{Polina Evdokimova}
\address{Lomonosov Moscow State University
	National Research University Higher School of Economics, Moscow}
\email{polina.evdokimova@math.msu.ru}
\thanks{This work is an output of a research project implemented as part of the Basic Research Program at HSE University.} 
\subjclass[2020]{Primary 14R20, 13N15; \ Secondary 13A50, 14L30}
\keywords{Automorphism group, locally nilpotent derivation, ind-group, toric variety, trinomial variety.}
\begin{document}

\begin{abstract}
We say that a locally nilpotent derivations $\delta$ is maximal if there are no inequivalent locally nilpotent derivations that commute with $\delta$. The paper gives a description of isotropy groups of maximal homogeneous locally nilpotent derivations on affine toric varieties and on certain trinomial hypersurfaces. Moreover, the criteria for homogeneous locally nilpotent derivations to be maximal were obtained for these classes of varieties.
\end{abstract}
\maketitle
\section{Introduction}

Let $\KK$ be an algebraically closed field of characteristic zero. Consider an irreducible affine algebraic variety $X$ with an algebra of regular functions $B := \mathbb{K}[X]$. Let $\mathrm{LND}(B)$ denote the set of all {\it locally nilpotent derivations} (LND) of algebra $B$, namely the set of derivations $\delta\colon B\rightarrow B$ such that for any $f\in B$ there is a non-negative integer~$n$, satisfying the condition $\delta^n(f)=0$.

There is a correspondence between LNDs and their {\it exponents}, where the exponent of an LND~$\delta$ is defined by the series $\exp(\delta) := \sum_{i \geqslant 0} \frac{\delta^i}{i!}$. Applying $\exp(\delta)$ to any element $b\in B$, one obtains a finite sum since $\delta$ is locally nilpotent. Hence, an exponent $\exp(\delta)$ is correctly defined. It is easy to show that $\exp(\delta)$ is an automorphism of the algebra $B$. Consider algebraic subgroups of the group $\mathrm{Aut}(B)$ of automorphisms of the algebra $B$, isomorphic to the additive group of the ground field $\KK$. We shall call such subgroups $\mathbb{G}_a$-subgroups. Each LND $\delta$ corresponds to the following subgroup $\mathcal{H}_\delta=\{ \exp(t\delta), t \in\KK \}$. This correspondence sets a bijection between LNDs, that are considered up to be proportionality with a scalar factor, and $\mathbb{G}_a$-subgroups in $\Aut(B)$. 

There is a natural action of the group $\Aut(X)$ on $\mathrm{LND}(B)$ by conjugation. Denote the stabilizer of an LND $\delta$ under this action by $\Aut(B)_\delta$. We shall call this stabilizer the isotropy subgroup of the LND. The isotropy subgroups were studied in several papers. For example, it was proved in~\cite{MP} that an isotropy subgroup of a simple derivation on a polynomial algebra in two variables is trivial. Stabilizers of Shamsuddin derivations were also studied in this paper. The isotropy subgroups of LNDs on some Danielewski surfaces and on some almost rigid varieties were described in~\cite{BV} and in~\cite{DL} respectively. The paper~\cite{DGL} concentrates on studying isotropy subgroups of LNDs on polynomial algebra in three variables.

There is an ind-group structure on a group of automorphisms $\Aut(B)$. The concept of ind-group traces back to I. R. Shafarevich's works, who named this object infinite-dimensional group, see~\cite{S} and~\cite{S1}. A survey of the results concerning the ind-groups can be found in~\cite{FK}. An isotropy subgroup of an LND is a natural closed in Zariski topology subgroup of a group $\Aut(B)$. Hence, ind-group theory can be used to study $\Aut(B)_\delta$.

Consider automorphisms of an algebra $B$, contained in $\mathbb{G}_a$-subgroups. We say that such automorphisms are {\it unipotent}. Two LNDs are called {\it equivalent} if their kernels coincide. It is known that equivalent LNDs commute. We shall call an LND $\delta$ {\it maximal} if every LND of an algebra~$B$, commuting with $\delta$, is equivalent to~$\delta$. For any LND consider a subgroup $\mathcal{U}(\delta)$ in $\mathrm{Aut}(B)$, consisting of exponents of all LNDs, that are equivalent to $\delta$. It follows from~\cite[Proposition 4.1]{RS} that LND $\delta$ is maximal if and only if a group $\mathcal{U}(\delta)$ is unique maximal commutative unipotent subgroup containing~$\exp(\delta)$.

If an automorphism of an algebra $B$ is contained in $\Aut(B)_\delta$, then it can be restricted to a kernel $\Ker(\delta)$ of this derivation. This restriction gives a natural homomorphism $\Theta\colon \Aut(B)_\delta\rightarrow \mathrm{Aut}(\mathrm{Ker}(\delta))$. It was proved in~\cite{DGL} that $\Ker(\Theta)$ coincides with $\mathcal{U}(\delta)$.

This paper presents a technique to describe stabilizers of homogeneous maximal LNDs on varieties with a torus action. The method is based on studying $\mathrm{Im}(\Theta)$. Consider subgroups of the group $\mathrm{Aut}(B)$, which are closed in ind-topology and satisfy some conditions. We prove that all maximal tori in such subgroups are conjugate to each other with respect to unipotent element, see Theorem~\ref{ind}. This Theorem can be applied to $\Aut(B)_\delta$ and a maximal LND $\delta$. Therefore, we obtain the following fact: automorphisms, contained in $\Aut(B)_\delta$, send elements of $\Ker(\delta)$ that are homogeneous with respect to a maximal torus to homogeneous elements, see Proposition~\ref{stab}.

This technique allows describing a group $\Aut(B)_\delta$ for homogeneous maximal LNDs on affine toric varieties, see Theorem~\ref{toric}. Theorems~\ref{trinomial} and~\ref{trinomial1} give a description of this group in the case where the varieties are trinomial hypersurfaces that are given by the equations of a special form.

The conditions for homogeneous LNDs on these varieties to be maximal are also studied. In the toric case a criterion in combinatorial terms is obtained, see Proposition~\ref{kriteriylnd}. Several types of trinomial hypersurfaces such that every homogeneous LND on these hypersurfaces is maximal, are found in Section~\ref{razdel9}. Moreover, maximal LNDs on various trinomial hypersurfaces are described.

The authors are grateful to S. Gaifullin for statement of the problem and continuous assistance during the work on the paper, to A. Perepechko for useful discussions and to I. Arzhantsev for valuable comments.

\section{Locally nilpotent derivations}

Some preliminaries on the field of locally nilpotent derivations are gathered in this section. Detailed information on this area can be found, for example, in~\cite{F}.

Let $B = \KK[X]$ be an algebra of regular functions on an irreducible affine variety $X$.

\begin{definition}
	A {\it derivation} of the algebra $B$ is a linear operator $\delta\colon B\rightarrow B$, satisfying the Leibniz rule:
	$\delta(fg)=f\delta(g)+g\delta(f)$.
\end{definition}
\begin{definition}
	A derivation $\delta\colon B\rightarrow B$ is called {\it locally nilpotent} (LND) if for every $f\in B$ there is a non-negative integer $n$ such that $\delta^n(f)=0$.
\end{definition}

Consider a grading on the algebra $B$ by an abelian group $G$
$$
B=\bigoplus_{g\in G}B_g.
$$
A derivation $\delta$ is called homogeneous if it sends homogeneous elements to homogeneous ones. It follows from the Leibniz rule that for any nonzero homogeneous derivation $\delta$ there is an element $g_0\in G$ such that $\delta$ sends $B_g$ to $B_{g+g_0}$. An element $g_0$ is called the {\it degree} of the derivation $\delta$. It is easy to prove that any derivation can be decomposed into a finite sum of homogeneous ones. We call these homogeneous summands {\it homogeneous components} of the derivation.

Consider a grading, given by a group $\mathbb{Z}$, and LND $\delta$ of $B$. Then $\delta=\sum\limits_{i=l}^k\delta_i$. It was proved in~\cite{R} that the components $\delta_l$ and $\delta_k$ with the minimal and the maximal degrees are also locally nilpotent. Now, suppose there is a $\mathbb{Z}^n$-grading and $\delta$ is decomposed into a sum of homogeneous derivations with respect to the grading. Then the convex hull of the degrees of homogeneous components forms a polyhedron. In this case it turns out that the components that correspond to the vertices of the polyhedron are locally nilpotent derivations.

Gradings on $B=\KK[X]$ by a free abelian group $\mathbb{Z}^n$ are in natural bijection with algebraic action of a torus $\TT=(\mathbb{K}^\times)^n$ on $X$. A lattice $\mathbb{Z}^n$ is identified with the character group $\mathfrak{X}(\TT)$ of the torus~$\TT$. An LND $\delta$ is $\mathbb{Z}^n$-homogeneous if and only if the torus $\TT$ is contained in the normalizer of a corresponding $\mathbb{G}_a$-subgroup of $\mathcal{H}_\delta$. We will frequently use the term $\TT$-homogeneous derivation instead of the term $\mathfrak{X}(\TT)$-homogeneous derivation.

\begin{definition}
	An LND $\delta$ is called {\it irreducible} if the set $\delta (B)$ is not contained in any proper principal ideal of $B$.
\end{definition}

\begin{definition}
	An element $r \in B$ is called {\it local slice} of an LND $\delta$ if $\delta(r) \neq 0$, but $\delta^2(r) = 0$.
\end{definition}

Denote the transcendence degree of a finitely generated $\KK$-algebra $R$ over $\KK$ by $\mathrm{tr.deg.}R$. Also denote localization of $R$ at an element $r$ by $R_r$ for any $r \in R$.

Fix an arbitrary LND $\delta$ and introduce the notation $A := \Ker(\delta)$. Then $A$ is a factorially closed subalgebra of $B$. Moreover, the following proposition is true:
\begin{proposition}\label{slice} \cite[Principle 11 (d),(e)]{F}
	Let $r \in B$ be an arbitrary local slice then $B_{\delta(r)} = \nonumber \\ = A_{\delta(r)} [r]$. In particular, $\mathrm{tr.deg.}A = \mathrm{tr.deg.}B - 1$.
\end{proposition}

It is easy to see that for any LND $\delta$ a derivation $h\delta$, where $h \in \Ker(\delta)$, is also locally nilpotent. Such derivations we call {\it replicas} of the derivation $\delta$.

\begin{definition}
	LNDs $\delta$ and $\partial$ are called {\it equivalent} if $\Ker(\delta) = \Ker(\partial)$. We use the notation $\delta \sim \partial$ for any two equivalent LNDs.
\end{definition}

It follows from~\cite[Principle 12]{F} that the equality $\Ker(\delta) = \Ker(\partial)$ is equivalent to the equality $\partial = h\delta$ for some $h$ from the quotient field $\mathrm{Quot}(\mathrm{Ker}(\delta))$. In particular, equivalent LNDs commute. Hence, for an LND $\delta$ one can consider the set $\mathcal{U}(\delta) := \{\exp(\partial) \;|\; \partial \sim \delta \}$, which is an abelian subgroup of the group $\Aut(B)$.

The next lemma is a particular case of~\cite[Lemma 3.2]{DG}:

\begin{lemma}\label{sumdiff}
	Suppose $D$ and $\delta$ are LNDs of an algebra $B$, $\partial$ is an LND, $\partial \sim \delta$, $D$ and $\delta$ commute. Then $D + \partial$ is also an LND.
\end{lemma}

Consider an LND with the following property:

\begin{definition}
	We say that an LND $\delta$ of an algebra $B$ is {\it maximal} if every LND of the algebra~$B$ that commutes with $\delta$ is equivalent to $\delta$.
\end{definition}

To state the equivalence of a homogeneous LND $\delta$ and all LNDs commuting with $\delta$, it suffices to check the equivalence of $\delta$ and all homogeneous LNDs commuting with $\delta$. The following proposition states this fact:

\begin{proposition}\label{odnor}
	Given a $\ZZ^n$-grading on the algebra $B$, consider a homogeneous LND $\delta$ such that every homogeneous LND commuting with $\delta$ is equivalent to it. Then $\delta$ is maximal.
\end{proposition}

\begin{proof}
	Let $[\delta, D] = 0$ for an LND $D$. Decompose $D$ into the sum of homogeneous components: $D = \sum_i D_i$. Obtain $[\delta, D_i] = 0$ for every $i$, since $[\delta, D] = \sum_i [\delta, D_i]$ and all derivations $[\delta, D_i]$ have distinct degrees.
	
	If $D$ is homogeneous, then the assertion follows. Otherwise, consider a homogeneous LND $D_k$ in the decomposition of $D$. It is equivalent to $\delta$ under the hypothesis of the proposition. Then $D - D_k$ is an LND by Lemma~\ref{sumdiff}. In addition, it commutes with $\delta$.
	
	Continuing in a similar manner, we see that every homogeneous component of $D$ is a homogeneous LND commuting with $\delta$, so they are equivalent to $\delta$. Hence, $D$ and $\delta$ are equivalent.
\end{proof}

\section{Ind-groups}
We begin with the necessary information on ing-groups and their Lie algebras following the paper~\cite{FK}.

\begin{definition}
	An {\it ind-variety} is a union $\mathcal{V} = \bigcup_{i \in \mathbb{N}} \mathcal{V}_i$ of an ascending sequence of algebraic varieties $\mathcal{V}_1 \hookrightarrow \mathcal{V}_2 \hookrightarrow \ldots \,$~such that all inclusions $\mathcal{V}_i \hookrightarrow \mathcal{V}_{i+1}$ are closed in the Zariski topology. If all~$\mathcal{V}_i$ are affine, then $\mathcal{V}$ is an {\it affine ind-variety}.
\end{definition}
An ind-variety has a natural {\it Zariski topology}: $U \subseteq \mathcal{V}$ is closed in~$\mathcal{V}$ if and only if $U \cap \mathcal{V}_i$ is closed in $\mathcal{V}_i$ for every $i$. It follows that a closed subset $U$ of the ind-variety~$\mathcal{V}$ has a natural structure of an ind-variety given by a sequence $U \cap \mathcal{V}_i$.

A {\it morphism} between two ind-varieties $\mathcal{V} = \bigcup_{i} \mathcal{V}_i$ and $\mathcal{W} = \bigcup_{i} \mathcal{W}_i$ is a map $\phi\colon \mathcal{V}\rightarrow \mathcal{W}$ such that for any $k$ there is an $l$ such that $\phi(\mathcal{V}_k) \subseteq \mathcal{W}_l$ and the induced map $\mathcal{V}_k\rightarrow \mathcal{W}_l$ is a morphism of algebraic varieties. An {\it isomorphism} of ind-varieties is a bijective morphism such that the inverse map is also a morphism.

\begin{definition}
	The {\it algebra of regular functions} $\KK[\mathcal{V}]$ on an affine ind-variety $\mathcal{V}$ is the projective limit $\varprojlim_{i} \KK[\mathcal{V}_i]$.
\end{definition}

The product $\mathcal{V} \times \mathcal{W}$ has a natural structure of an ind-variety given by a sequence $\mathcal{V}_1 \times \mathcal{W}_1 \hookrightarrow \mathcal{V}_2 \times \mathcal{W}_2 \hookrightarrow \ldots$

\begin{definition}
	An {\it ind-group} $\mathcal{G}$ is an ind-variety with a group structure such that multiplication $\mathcal{G} \times \mathcal{G}\rightarrow \mathcal{G}$, $(g,h) \mapsto gh$, and the inverse $\mathcal{G} \rightarrow \mathcal{G}$, $g \mapsto g^{-1}$, are morphisms of ind-varieties.
\end{definition}

A closed subgroup of an ind-group is also an ind-group with respect to the natural structure of an ind-variety.

\begin{definition}
	An ind-group $\mathcal{G}$ is called {\it nested} if $\mathcal{G} = \bigcup_i \mathcal{G}_i$, where $\mathcal{G}_i$ is an algebraic group and $\mathcal{G}_i$ is closed in $\mathcal{G}_{i+1}$ for all $i$.
\end{definition}

The automorphism group of an affine variety $X$ has a natural structure of an ind-group, see~\cite[Section~5]{FK}.

For any ind-variety $\mathcal{V} = \bigcup_{i} \mathcal{V}_i$, we can define the {\it tangent space} in a point $v \in \mathcal{V}$: since $v \in \mathcal{V}_k$ for $k \geq k_0$ for some $k_0$ and $T_v \mathcal{V}_i \subseteq T_v \mathcal{V}_j$ for $j \geq i \geq k_0$, so
\begin{equation*}
	T_v\mathcal{V} := \bigcup_{k \geq k_0} T_v \mathcal{V}_i.
\end{equation*}
If $\mathcal{G}$ is an affine ind-group, we can naturally identify the tangent space of $\mathcal{G}$ at the identity $e \in \mathcal{G}$ with the set of all left-invariant (commuting with all homomorphisms $\lambda^{*}_g$ that are induced by left translations $\lambda_g\colon \mathcal{G} \rightarrow \mathcal{G}, h \mapsto gh$) derivations of the algebra $\KK[\mathcal{G}]$, see~\cite[Section~2]{FK}. This identification allows us to carry over a structure of Lie algebra to $T_e \mathcal{G}$. Thus, we come to the following definition:
\begin{definition}
	The {\it Lie algebra} $\Lie(\mathcal{G})$ of an affine ind-group $\mathcal{G}$ is the set of all left-invariant derivations of the algebra $\KK[\mathcal{G}]$ with the standard Lie bracket: $[\delta_1, \delta_2] := \delta_1 \delta_2 - \delta_2 \delta_1$.
\end{definition}

By an {\it action} of an ind-group $\mathcal{G}$ on an ind-variety $\mathcal{V}$ we mean a homomorphism $\rho\colon \mathcal{G} \rightarrow \Aut(\mathcal{V})$ such that the action map $\mathcal{G} \times \mathcal{V} \rightarrow \mathcal{V}$ is a morphism of ind-varieties.

Let $V$ be a vector space of at most countable dimension. Then $V$ is a union of ascending finite-dimensional subspaces, which endows $V$ with the structure of an ind-variety. Denote the Lie algebra of all endomorphisms of the vector space $V$ by $\mathcal{L}(V)$.

Consider a linear action of an ind-group $\mathcal{G}$ on $V$, that is, a homomorphism $\rho \colon \mathcal{G} \rightarrow \mathrm{GL}(V)$. For each $v \in V$, let $\mu_v \colon \mathcal{G} \rightarrow V$, $g \mapsto gv$, denote the orbit map.  Then we can define $d\rho \colon \Lie(\mathcal{G}) \rightarrow \mathcal{L}(V)$ as follows:
\begin{equation*}
	d\rho(\eta) (v) = d_e \mu_v(\eta),
\end{equation*}
where $\eta \in \Lie(\mathcal{G})$. According to~\cite[Lemma 2.6.2]{FK}, the map $d\rho$ is a homomorphism of Lie algebras; thus, we obtain an action of the Lie algebra $\Lie(\mathcal{G})$ on $V$. Henceforth, for a linear action of an ind-group on a vector space, we will call the action obtained by this construction the Lie algebra action.

Let $\mathrm{Vec}(X)$ denote the Lie algebra of vector fields on an affine variety $X$. As is known, this Lie algebra is identified with the Lie algebra $\mathrm{Der}(\KK[X])$ of derivations of $\KK[X]$. Suppose that an ind-group $\mathcal{G}$ acts on an affine variety $X$, and for any $x \in X$ let $\mu_x\colon \mathcal{G} \rightarrow X$, $g \mapsto gx$, be the orbit map. Then each element $\eta \in \Lie(\mathcal{G})$ defines a vector field $\xi_\eta \in \mathrm{Vec}(X)$ as follows:
\begin{equation*}
	\xi_\eta(x) := d_e\mu_x(\eta).
\end{equation*}
We can identify $\Lie(\Aut(X))$ with the Lie subalgebra of $\mathrm{Vec}(X)$ (or of $\mathrm{Der}(\KK[X])$) due to the following fact:
\begin{proposition}\cite[Proposition 7.2.4]{FK}. 
	The map $\xi \colon \Lie(\mathcal{G})\rightarrow \mathrm{Vec}(X)$ is an anti-homomorphism of Lie algebras. If $\mathcal{G} = \Aut(X)$, then $\xi$ is injective.
\end{proposition}

\begin{definition}
	A derivation $\delta \in \mathrm{Der}(\KK[X])$ is called {\it semisimple} if there exists a basis $\{ f_i \; | \; i \in \mathbb{N} \}$ of the vector space $\KK[X]$ such that $\delta(f_i) \in \KK f_i \;\; \forall i \in \mathbb{N}$.
\end{definition}
\begin{definition}
	A derivation $\delta \in \mathrm{Der}(\KK[X])$ is called {\it locally finite} if it acts locally finitely on $\KK[X]$, i.e., for any~$f \in \KK[X]$ there is a finite-dimensional $\delta$-invariant vector subspace $V\subset\KK[X]$ such that $f \in V$.
\end{definition}

It is easy to see that both semisimple and locally nilpotent derivations are locally finite. Moreover, any locally finite derivation $\delta$ admits the Jordan decomposition $\delta = \delta_s + \delta_n$, where $\delta_s$ is a semisimple derivation and $\delta_n$ is an LND. Here, $\delta_s$ and $\delta_n$ commute, and any derivation that commutes with $\delta$ also commutes with both $\delta_s$ and $\delta_n$, see~\cite[Section 2]{FZ}.

The following statement appears in various works, for instance, in~~\cite[Lemma 3.1]{FZ}, \cite[Lemma~4.1]{PR} or~\cite[Theorem 2.1]{AG}, and generalizes the statement for LNDs presented in the previous section:
\begin{lemma}\label{razl}
	Let $\delta = \sum_i \delta_i$  be the decomposition of a locally finite derivation into its homogeneous components. If $\delta \neq \delta_0$, then there exists a locally nilpotent component $\delta_k$ with $k \neq 0$.
\end{lemma}

Until the end of this section, $X$ is an affine algebraic variety, $\Aut(X)$ is its automorphism group with the structure of an ind-group, $\mathcal{G} \subseteq \Aut(X)$ is a closed ind-subgroup, and $\mathfrak{g} := \Lie(\mathcal{G})$ is its Lie algebra.

Following~\cite[Section~4]{PR}, we prove a generalization of~\cite[Proposition~4.4]{PR} for certain closed subgroups of $\Aut(X)$.

\begin{lemma}\label{ind-lem}
	Assume that all LNDs in $\mathfrak{g}$ are equivalent. Suppose that for any locally finite derivation $\delta \in \mathfrak{g}$, we have $\delta_s$, $\delta_n \in \mathfrak{g}$, where $\delta = \delta_s + \delta_n$ is the Jordan decomposition of $\delta$. If $\delta$ is a locally finite derivation in $\mathfrak{g}$ and $\partial$ is a locally nilpotent derivation in $\mathfrak{g}$, then~$\delta - \partial$ is locally finite.
\end{lemma}

\begin{proof}
	
	Since $\delta$ is a locally finite derivation, its semisimple component $\delta_s$ and locally nilpotent component $\delta_n$ from the Jordan decomposition both belong to $\mathfrak{g}$.
	
	Note that for any LND $\partial'$, we have [$\delta, \partial'$] = [$\delta_s, \partial'$] since [$\delta_n, \partial'$] = 0 by the assumption of the lemma.
	
	Consider the decomposition of $\partial$ into homogeneous components with respect to $\delta_s$:
	\begin{equation*}
		\partial = \sum_i \partial_i, \;\; \text{where} \;\; [\delta_s, \partial_i] = i \partial_i, \;\; i \in \ZZ.
	\end{equation*}
	Then we get:
	\begin{equation}
		\begin{split}
			\begin{aligned}\label{vander}
				\partial = \sum_i \partial_i \in \mathfrak{g}, \\
				[\delta_s,\partial] = \sum_i i \partial_i \in \mathfrak{g}, \\
				[\delta_s,[\delta_s,\partial]] = \sum_i i^2 \partial_i \in \mathfrak{g}, \\
				\ldots
			\end{aligned}
		\end{split}
	\end{equation}
	Thus, these commutators are expressed in terms of $\partial_i$ via the system of equations with a Vandermonde matrix. Since all the indices $i$ are distinct, we can solve the system~(\ref{vander}) for $\partial_i$, therefore, all $\partial_i \in \mathfrak{g}$.
	
	If $\partial = \partial_0$, then [$\delta, \partial$] = 0, and the difference between two commuting locally finite derivations is again locally finite.
	
	If $\partial \neq \partial_0$, then, by Lemma~\ref{razl}, there exists a homogeneous locally nilpotent component $\partial_\upsilon$ in its decomposition. Therefore, $[\partial,\partial_\upsilon]$ = 0, by the hypothesis, so~$\partial - \partial_\upsilon$ is again an LND. Continuing this process, we obtain the following: all nonzero homogeneous components in the decomposition of $\partial$ are LNDs. Hence, $\partial_i = h_i \partial$,  where all $h_i$ belong to the fraction field of the kernel of $\partial$.
	
	Let $\partial' = \sum_{i \neq 0} \frac{1}{i} \partial_i =  \sum_{i \neq 0} \frac{h_i}{i}\partial$. Then: 
	\begin{equation*}
		[\delta, \partial'] = \sum_{i \neq 0} h_i\partial = \partial - \partial_0.
	\end{equation*}
	The derivations $[\delta, \partial']$ and $\partial'$ are locally nilpotent and thus commute. By~\cite[Lemma 2.4]{FZ}, we have that $\delta - [\delta, \partial'] = \delta - \partial + \partial_0 = \mathrm{exp}(\partial') \circ \delta \circ \mathrm{exp}(-\partial')$ is locally finite. Since
	$\partial_0$ commutes with both $\delta$ and $\partial \, - \, \partial_0$, the difference between the locally finite derivations $\delta \, - \, \partial \, + \, \partial_0$ and $\partial_0$ is again a locally finite derivation.
\end{proof}

\begin{theorem}\label{ind}
	Let $X$ be an affine variety, $\Aut(X)$ be its automorphism group, and  $\mathcal{G} \subseteq \Aut(X)$ be a subgroup closed in the ind-topology. 
	Assume that for $\mathfrak{g} := \Lie(\mathcal{G})$ the following conditions hold:
	\begin{enumerate}
		\item[1)] all LNDs in $\mathfrak{g}$ are equivalent,
		\item[2)] for any LND $\partial \in \mathfrak{g}$, we have $\mathrm{exp}(\partial) \in \mathcal{G}$,
		\item[3)] for any locally finite derivation $\delta \in \mathfrak{g}$, we have $\delta_s$, $\delta_n \in \mathfrak{g}$, where $\delta = \delta_s + \delta_n$ is the Jordan decomposition of $\delta$.
	\end{enumerate}
	Then all maximal tori in $\mathcal{G}$ are conjugate by $\mathrm{exp}(\partial)$ for some LND $\partial \in \mathfrak{g}$.
\end{theorem}

\begin{proof}
	
	Let $\mathbb{T}_1$ and $\mathbb{T}_2$ be two distinct maximal tori in $\mathcal{G}$. Let $\Lambda$ be a semisimple derivation corresponding to a one-dimensional subtorus $\mathbb{T'}_2$ of the torus $\mathbb{T}_2$. 
	
	Consider the $\mathbb{Z}^n$-grading defined by the torus $\mathbb{T}_1$ and decompose $\Lambda$ into homogeneous components with respect to this grading: $\Lambda = \sum_{i \in \ZZ^n} \Lambda_i$. Then, by Lemma~\ref{razl}, there exists a nonzero $l \in \ZZ^n$ such that $\Lambda_l$ is a nonzero LND. Hence, by Lemma~\ref{ind-lem}, \mbox{$\Lambda - \Lambda_l$} is a locally finite derivation. The decomposition of $\Lambda - \Lambda_l$ contains less nonzero components. Continuing to reduce the number of homogeneous components in the same manner, we eventually arrive at a locally finite derivation $\Lambda_0$ that commutes with $\TT_1$.
	
	Consider the set of LNDs $\Omega := \{ \Lambda_i, \; i \neq 0 \}$ $\cup$ $\{ (\Lambda_0)_n \}$, where $(\Lambda_0)_n$ is the locally nilpotent component of $\Lambda_0$ and let $\mathfrak{u}$ denote the linear span of $\Omega$. All LNDs from $\Omega$ belong to $\mathfrak{g}$ and hence commute, therefore, $\mathfrak{u}$ is a Lie algebra. Moreover, the set $\Omega$ is finite, so the group $U := \, =~\{ \exp(\partial) \; | \; \partial \in \mathfrak{u} \} \subseteq \mathcal{G}$ is algebraic.
	
	Now, consider the group generated by the torus $\mathbb{T}_1$ and the group $U$. Since $\mathbb{T}_1$ normalizes $U$, this group is algebraic and coincides with the group $\TT_1 \ltimes U \subseteq \mathcal{G}$.
	
	Since $\Lambda_0$ commutes with $\TT_1$, its semisimple component also commutes with $\TT_1$. Hence, by the maximality of $\TT_1$, this component belongs to $\mathfrak{t}_1 := \Lie(\TT_1)$. Thus, $\Lambda_0 \in \mathfrak{t}_1 \ltimes \mathfrak{u} = \Lie(\TT_1 \ltimes U)$. Consequently, $\Lambda$ belongs to $\mathfrak{t}_1 \ltimes \mathfrak{u}$ and by~\cite[Remark 17.3.3]{FK} the torus $\TT'_2$ is contained in the group $\TT_1 \ltimes U$. 
	
	Let $\mathbb{T}_2 = \mathbb{T}_{21} \times \dots \times \mathbb{T}_{2n}$, where each $\mathbb{T}_{2i}$ is a one-dimensional torus. For every torus $\TT_{2i}$ we carry out the above procedure. We obtain an inclusion $\TT_{2i} \subseteq \TT_1 \ltimes U_i$, where $U_i$ is a commutative unipotent algebraic group. Denote the subgroup generated by all $U_i$ by $U$. Then all $\TT_{2i}$ are contained in the group $\TT_1 \ltimes U$. Therefore, $\TT_2 \subseteq \TT_1 \ltimes U$. Since the group $\TT_1 \ltimes U$ is algebraic, its maximal tori $\TT_1$ and $\TT_2$ are conjugated by an element of the form $\exp({\delta})t$, where $t \in \TT_1$ and $\delta$ is an LND. Therefore, these tori are also conjugated by an element $\exp({\delta})$.
\end{proof}

\section{Isotropy subgroups}

There is a natural action of the automorphism group $\Aut(B)$ of the algebra~$B$ on the set of all LNDs of $B$ by conjugation. Denote the stabilizer (or isotropy subgroup) of an LND $\delta$ under this action by $\Aut(B)_\delta := \{ \phi \in \Aut(B) \,|\, \phi \delta = \delta \phi \}$. The group $\Aut(B)_\delta \subset \Aut(B)$ is closed in the ind-topology, since its intersection with every finite-dimensional variety, such that the union of these varieties constitutes $\Aut(X)$, is closed in the variety.

\begin{remark}\label{zamind}
	The action of the Lie algebra $\Lie(\Aut(B)_\delta)$ on $\langle \delta \rangle$ is trivial, since the action of the group $\Aut(B)_\delta$ by conjugation on the linear space $\langle \delta \rangle$ is trivial. Moreover, for any element $\eta  \in \Lie(\Aut(X))$ we have $\eta \cdot \delta = [\delta, \xi_\eta]$ by~\cite[Section 7.3]{FK}. It implies that all derivations belonging to~$\Lie(\Aut(B)_\delta)$ commute with~$\delta$.
\end{remark}

It follows from Remark~\ref{zamind} that all LNDs in $\Lie(\Aut(B)_\delta)$ commute with $\delta$, so their exponents also commute with $\delta$. Hence, condition~(2) of Theorem~\ref{ind} is satisfied for $\Aut(B)_\delta$. The properties of the Jordan decomposition ensure that condition~(3) is also satisfied for this group.

Now, let $\delta$ be a maximal LND. This implies that condition~(1) is satisfied. Therefore, we can apply Theorem~\ref{ind} to the group $\Aut(B)_\delta$ and obtain the following:

\begin{proposition}\label{stab}
	
	Let $\TT$ be a maximal torus in $\Aut(B)_\delta$. Assume that an LND $\delta$ is maximal. Then any $\phi \in \Aut(B)_\delta$ preserves the set of $\TT$-semi-invariants of the algebra $A := \Ker(\delta)$.
	
\end{proposition}

\begin{proof}
	
	Let $f \in A$ be a $\TT$-homogeneous element. All maximal tori in $\Aut(B)_\delta$ are conjugate by $\exp(h\delta)$ for some $h \in \mathrm{Quot}(A)$ by Theorem~\ref{ind}. Consequently:
	\begin{equation*}
		(\exp(h\delta) \circ t \circ \exp(-h\delta)) \cdot f = (\exp(h\delta) \circ t) \cdot f = \exp(h\delta) \cdot (\lambda f) = \lambda f,
	\end{equation*}
	where $t \in \TT$ and $\lambda \in \KK$. Thus, $f$ is semi-invariant under the action of all maximal tori in $\Aut(B)_\delta$. Then for any $t \in \TT$
	\begin{equation*}
		t \cdot \phi (f) = (\phi \circ \phi^{-1} \circ t \circ \phi) \cdot f = \phi (\lambda f) = \lambda \phi(f),
	\end{equation*}
	where the second equality follows from the fact that $\phi^{-1}t\phi$ is an element of the maximal torus $\phi^{-1}\TT \phi$.
	
\end{proof}

Clearly, any $\phi \in \Aut(B)_\delta$ preserves the subalgebra $A$. Hence, we can consider the natural restriction homomorphism:
\begin{equation*}
	\Theta \colon \Aut(B)_\delta \rightarrow  \Aut(A).
\end{equation*}

The following fact on the homomorphism $\Theta$ was proved in the paper~\cite{DGL}. However, we also give the proof of this fact from that paper here as well:
\begin{proposition}\label{Ker}
	$\Ker(\Theta) = \mathcal{U}(\delta)$.
\end{proposition}

\begin{proof}
	The inclusion $\mathcal{U}(\delta) \subseteq \Ker(\Theta)$ is obvious, let us prove the reverse inclusion.
	
	Let $\psi \in \Ker(\Theta)$, that is, $\psi|_A = \mathrm{id}$. Fix a local slice $f \in B$ and denote $g := \delta(f)$. Then $\psi$ extends to an automorphism $\widetilde \psi$ of the algebra $B_g = A_g[f]$.
	
	Since $\widetilde \psi$ acts identically on $A_g$ and $f$ is transcendental over the integral domain $A_g$, we obtain $\widetilde \psi(f) = uf + v$, where $u, v \in A_g$.
	
	The derivation $\delta$ also extends to a derivation $\widetilde \delta$ of algebra $B_g$, and the following holds:
	\begin{equation*}
		g = \psi(g) = \psi(\delta(f)) = \delta(\psi(f)) = \widetilde \delta (uf + v) = u\delta(f) = ug.
	\end{equation*}
	
	So we obtain $u = 1$. Then $v = \psi(f) - f \in B$, and hence, $v \in A$. Define a derivation $\widetilde D$ of the algebra $B_g$ as follows: $\widetilde D|A_g = 0$ and $\widetilde D(f) = v$. It is easy to see that $\widetilde D$ can be restricted to a derivation $D$ of the algebra $B$ and $\psi = \mathrm{exp}(D)$. Since $A \subseteq \Ker(D)$ we conclude that $D \sim \delta$, which completes the proof.
\end{proof}

\section{Toric varieties}

Let us briefly recall basic facts about toric varieties. For more information on this field, see, for example,~\cite{CLS}, \cite{Ful}.

\begin{definition}
	An irreducible algebraic variety $X$ is called {\it toric} if there exists a regular action of a torus $\TT$ on $X$ with an open orbit.
\end{definition}

Note that in this definition X need not be normal. However, in this and the next section, we assume all toric varieties to be normal, whereas in Section~\ref{razdel7} we will consider toric varieties that are not necessarily normal.

In this and the next section, let $X$ be an affine normal toric variety with an effective action of an algebraic torus $\TT$ and $\dim X = n$. 

Let $N$ be the lattice of one-parameter subgroups of the torus $\TT$, $M := \mathrm{Hom}(N, \ZZ)$ be the dual lattice of characters and $\langle \cdot, \cdot\rangle \colon M \times N \rightarrow \ZZ$ be the natural pairing. This pairing can be extended to a pairing $\langle \cdot, \cdot\rangle \colon M_{\QQ} \times N_{\QQ} \rightarrow \QQ$ of the vector spaces $N_{\QQ} := N \otimes_\ZZ\QQ$ and $M_{\QQ} := M \otimes_\ZZ\QQ$.

Let $\chi^m$ denote the character of $\TT$ which corresponds to a lattice point  $m \in M$. The group algebra
\begin{equation*}
	\bigoplus_{m \in M} \KK \chi^m
\end{equation*}
can be identified with the algebra $\KK[\TT]$ of regular functions on the torus $\TT$. The algebra $\KK[X]$ is identified with a subalgebra of $\KK[\TT]$ as follows: one assigns a rational polyhedral cone $\sigma$ in the vector space $N_{\QQ}$ to the variety $X$. The cone $\sigma$ corresponds to the dual cone $$\sigma^\vee := \{ m \in M_\QQ \; | \; \langle m, v \rangle \geqslant 0 \;\; \forall v \in \sigma \}$$ in the vector space $M_{\QQ}$, and $\KK[X]$ coincides with a semigroup algebra of the semigroup $\sigma^\vee \cap M $:
\begin{equation*}
	\KK[X] := \bigoplus_{m \in \sigma^\vee \cap M} \KK \chi^m.
\end{equation*}

Since the action of the torus $\TT$ on $X$ is effective, the cone $\sigma^\vee$ is of full dimension or, equivalently, the cone $\sigma$ is pointed, i.e., contains no lines.

All $\mathbb{G}_a$-subgroups normalized by the torus $\TT$ in the automorphism group of a complete normal toric variety were described in~\cite{D}. It is known that $\TT$-normalizable $\mathbb{G}_a$-subgroups in $\Aut(X)$ correspond to $\TT$-homogeneous LNDs on $\KK[X]$ in the case of an affine normal toric variety $X$. The description of such LNDs was obtained in~\cite{L}. Let us briefly recall this description.

Let $\rho_1, \ldots, \rho_k$ be all the extremal rays of the cone $\sigma$, and $v_1, \ldots, v_k \in N$ be the primitive vectors on these rays.
\begin{definition}
	A {\it Demazure root} of cone $\sigma$ corresponding to $\rho_i$ is a vector $e \in M$ such that
	\begin{equation*}
		\langle e, v_i \rangle = -1, \;\;\;\; \langle e, v_j \rangle \geqslant 0 \;\;\; \forall j \neq i.
	\end{equation*}
\end{definition}
The Demazure root $e$ defines the $\TT$-homogeneous LND $\delta_e$ on $\KK[X]$ by the formula:
\begin{equation*}
	\delta_e(\chi^m) = \langle e, v_i \rangle \chi^{m+e}.
\end{equation*}

Moreover, there is a one-to-one correspondence between Demazure roots and all homogeneous LNDs up to proportionality with a scalar factor.

The kernel of a homogeneous LND $\delta_e$ is a finitely generated subalgebra of $\KK[X]$ and has the following form:
\begin{equation*}
	\Ker(\delta_e) = \bigoplus_{m \in \rho_i^\bot \cap M} \KK \chi^m,
\end{equation*}
where $\rho_i^\bot := \{ m \in M \; | \; \langle m, v_i \rangle = 0 \}$. In particular, two homogeneous LNDs $\delta_e$ and $\delta_{e'}$ are equivalent if and only if $e$ and $e'$ correspond to the same extremal ray of the cone $\sigma$.

\section{Commuting homogeneous LNDs on toric varieties}

In this section, we describe a certain condition for two extremal rays of the cone $\sigma$ in combinatorial terms. This condition guarantees that the corresponding commuting homogeneous LNDs exist. We also obtain a combinatorial criterion for a homogeneous LND to be maximal.

The following lemma describes in combinatorial terms the condition for the commutativity of two homogeneous LNDs on a toric variety. It is fairly well-known, but, for the sake of completeness, we present its proof:

\begin{lemma}\label{isvtorlem}
	Let the Demazure roots $e$ and $e'$ correspond to distinct extremal rays $\rho$ and $\rho'$. Then
	\begin{equation*}
		[\delta_e, \delta_{e'}] = 0 \iff \langle e, v' \rangle = 0 \;\;\text{and} \;\; \langle e', v \rangle = 0.
	\end{equation*}
\end{lemma}
\begin{proof}
	For any homogeneous function $\chi^m$, we have:
	\begin{equation*}
		\delta_{e'}(\delta_e(\chi^m)) = \delta_{e'}(\langle m, v \rangle \chi^{m+e}) = \langle m+e, v' \rangle \langle m, v \rangle \chi^{m+e+e'}.
	\end{equation*}
	
	Similarly, $\delta_{e}(\delta_{e'}(\chi^m))  = \langle m+e', v \rangle \langle m, v' \rangle \chi^{m+e+e'}$.
	
	It follows that the commutator equals zero if and only if the equality $\langle e, v' \rangle \langle m, v \rangle = \langle e', v \rangle \langle m, v' \rangle$ holds or, equivalently, $\langle m, \langle e, v' \rangle v - \langle e', v \rangle v' \rangle = 0$ for all $m \in M \cap \sigma^{\vee}$. We may assume that the previous equality holds for all $m \in M$, due to the bilinearity of the pairing and the fact that the linear span of the cone $\sigma^\vee$ coincides with $M_\QQ$. Thus, this equality is equivalent to
	\begin{equation*}
		\langle e, v' \rangle v - \langle e', v \rangle v' = 0.
	\end{equation*}
	
	Since the vectors $v$ and $v'$ are not proportional, the previous equality holds simultaneously with the equalities $\langle e, v' \rangle = 0$ and $\langle e', v \rangle = 0$, as required.
\end{proof}

Let us prove the following technical lemma before proceeding with the proof of the combinatorial conditions:

\begin{lemma}\label{tex}
	Let $v$ and $v'$ be vectors in the lattice $N$. Then there exist $e$ and $e' \in M$ satisfying the following equalities:
	
	\begin{equation}\label{sistema}
		\langle e, v \rangle = -1, \;\; \langle e, v' \rangle = 0, \;\; \langle e', v \rangle = 0, \;\; \langle e', v' \rangle = -1
	\end{equation}
	if and only if the vectors $v$, $v'$ can be included in a basis of the lattice $N$.
	
\end{lemma}

\begin{proof}
	Suppose that the vectors $v$, $v'$ can be included in a basis of the lattice $N$. Then there exists an invertible matrix $P \in GL_n(\ZZ)$ such that its first two columns are $v$ and $v'$. Taking the first two rows of the matrix $-P^{-1}$ as $e$ and $e'$, from the equality $-P^{-1}P = -E$ we obtain the equalities (\ref{sistema}).
	
	Conversely, assume that the equalities (\ref{sistema}) are satisfied. Suppose that $v$, $v'$ cannot be included in a basis of the lattice $N$. Then consider a matrix such that its rows are the vectors $v$ and $v'$. All $2 \times 2$ - minors of the matrix are divisible by a number $d \neq \pm 1$ by~\cite[I.2.3, Lemma 2]{C}. Let $v = (a_1, \ldots, a_n)$, $v' = (b_1, \ldots, b_n)$ and $e = (e_1, \ldots, e_n)$. Then the following equalities hold:
	\begin{equation}\label{1}
		\sum \limits_{i = 1}^{n} a_ie_i = -1,
	\end{equation}
	\begin{equation}\label{2}
		\sum \limits_{i = 1}^{n} b_ie_i = 0.
	\end{equation}
	Multiplying the equality (\ref{2}) by $a_1$ and subtracting from it the equality (\ref{1}) multiplied by $b_1$, we obtain the following:
	\begin{equation}\label{3}
		\sum \limits_{i = 2}^{n} (a_1b_i - a_ib_1)e_i = b_1.
	\end{equation}
	Every term on the left-hand side of the equality (\ref{3}) is divisible by $d$, so $b_1$ is divisible by $d$. Similarly, it can be proved that all $b_i$ are divisible by $d$. Then it follows that the vector $v'$ is not primitive, hence, the equality $\langle e', v' \rangle = -1$ cannot hold. This contradiction completes the proof of the lemma.
\end{proof}

\begin{remark}\label{rem}
	If we remove one of the equalities $\langle e, v' \rangle = 0$ or $\langle e', v \rangle = 0$ from (\ref{sistema}), then, as is clear from the proof of Lemma~\ref{tex}, the vectors $v$, $v'$ will still extend to a basis of the lattice $N$.
\end{remark}

\begin{proposition}\label{commutelnd}
	Let $v$ and $v'$ be primitive vectors on distinct extremal rays $\rho$ and $\rho'$ of the cone~$\sigma$. Then there exist corresponding commuting homogeneous LNDs $\delta_e$ and $\delta_{e'}$ if and only if $\rho$ and $\rho'$ lie in a common two-dimensional face of the cone $\sigma$ and the vectors $v$, $v'$ can be included in a basis of the lattice $N$.
\end{proposition}

\begin{proof}
	Suppose that the corresponding homogeneous LNDs commute. This is equivalent to the following equalities: $\langle e, v' \rangle = 0$, $\langle e', v \rangle = 0$ by Lem\-ma~\ref{isvtorlem}. Together with the conditions that $e$ and $e'$ are Demazure roots, we obtain equalities~(\ref{sistema}). Then, by Lemma~\ref{tex}, the vectors $v$, $v'$ can be included in a basis of the lattice~$N$.
	
	The vector $v$ lies on an extremal ray of the cone $\sigma$, hence, there exists a supporting hyperplane that intersects $\sigma$ exactly along the edge $\rho$. That is, there exists $\alpha \in M$ (defining this hyperplane) such that the following holds:
	\begin{equation*}
		\langle \alpha, v \rangle = 0, \;\; \langle \alpha, v' \rangle = k > 0, \;\; \langle \alpha, v_i \rangle > 0,
	\end{equation*}
	where $v_i$ are the primitive vectors on the remaining extremal rays of the cone $\sigma$. Similarly, there exists $\beta \in M$ such that the following holds:
	\begin{equation*}
		\langle \beta, v' \rangle = 0, \;\; \langle \beta, v \rangle = l > 0, \;\; \langle \beta, v_i \rangle > 0.
	\end{equation*}
	
	Denote $\omega := \alpha + \beta + le + ke'$. Then $\langle \omega, v \rangle =\langle \omega, v' \rangle = 0$ and $\langle \omega, v_i \rangle > 0$, which means that $v$ and $v'$ lie on the same two-dimensional face of the cone $\sigma$.
	
	Now we prove the converse: we obtain equalities (\ref{sistema}) for some $\tilde{e}$, $\tilde{e}' \in M$ from the fact that $v$ and $v'$ can be included in a basis of the lattice $N$ by Lemma~\ref{tex}. However, $\tilde{e}$ or $\tilde{e}'$  may not be Demazure roots, since their pairings with the other vectors $v_i$ may be negative.
	
	The fact that $v$ and $v'$ lie on the same two-dimensional face of $\sigma$ implies the existence of an element $\omega$ such that $\langle \omega, v \rangle =\langle \omega, v' \rangle = 0$ and $\langle \omega, v_i \rangle > 0$ since the cone $\sigma$ is pointed. Then, for sufficiently large $k$ and $l \in \mathbb{N}$, the elements $e := \tilde{e} + k\omega$ and $e' := \tilde{e}' + l\omega$ become suitable Demazure roots.
\end{proof}

For a given extremal ray $\rho$, we denote the set of all extremal rays that lie with $\rho$ on the same two-dimensional face of the cone $\sigma$ by $\mathcal E(\rho)$ .

\begin{proposition}\label{kriteriylnd}
	Let $e$ be a Demazure root corresponding to an extremal ray $\rho$ of the cone $\sigma$. Then the LND $\delta_e$ is maximal if and only if $\langle e, v' \rangle \neq 0$ for every extremal ray $\rho' \in \mathcal E(\rho)$.
\end{proposition}

\begin{proof}
	Suppose that there exists an LND $\partial$ commuting with $\delta_e$ but this LND is not equivalent to~$\delta_e$. We may assume that $\partial = \delta_{e'}$ is a homogeneous LND by Proposition~\ref{odnor}. Then $\langle e, v' \rangle = 0$ and $\rho' \in \mathcal E(\rho)$ by Proposition~\ref{commutelnd}.
	
	Conversely, suppose that for some $\rho' \in \mathcal E(\rho)$ we have $\langle e, v' \rangle = 0$. Then, by Remark~\ref{rem} and Proposition~\ref{commutelnd}, there exists $e'$ such that $\delta_e$ commutes with $\delta_{e'}$, but they are not equivalent.
\end{proof}

\begin{remark}
	It is easy to see that a Demazure root $e$ corresponding to an extremal ray $\rho$ cannot have zero pairing with the primitive vector on an extremal ray $\rho' \notin \mathcal E(\rho)$. Indeed, suppose $\langle e, v' \rangle = 0$. Let $\alpha \in M$ define a supporting hyperplane separating the extremal ray $\rho'$, with $\langle \alpha, v \rangle = k > 0$. Then $\omega := \alpha + ke$ defines a hyperplane separating both $\rho$ and $\rho'$, hence, they lie on the same two-dimensional face of the cone $\sigma$.
\end{remark}

\section{Isotropy subgroups of homogeneous LNDs on toric varieties}\label{razdel7}

Let $X$ be an affine toric variety with an effective action of an algebraic torus $\TT$, $B := \KK[X]$ be the algebra of regular functions on $X$ and $\dim X = n$. 

Note that $X$ is not assumed to be normal. Every affine non-normal toric variety can also be associated with a semigroup such that its semigroup algebra coincides with the algebra of regular functions on this variety, see~\cite[Theorem 1.1.16]{CLS}. However, in contrast to the normal case, for a non-normal variety this semigroup may not have the form $\sigma^\vee \cap M$ for a rational polyhedral cone~$\sigma$. Nevertheless, a non-normal toric variety can be associated with a cone $\sigma$ corresponding to the normalization of this variety, see~\cite[Proposition 1.3.8]{CLS}.

In this setting, unlike the normal case, it may happen that not every Demazure root of the cone $\sigma$ corresponds to a homogeneous LND of the algebra $B$. Nevertheless, every homogeneous LND of $B$ is still of the form $c\delta_e$ for some Demazure root $e$ and $c \in \mathbb{K}$, see~\cite{BG} and~\cite{DiL}. Since LNDs that are proportional with some non-zero scalar multiplier have coinciding stabilizers, we may assume without loss of generality that $c = 1$.

Consider a homogeneous LND $\delta = \delta_e$, where $e = (e_1, \ldots, e_n)$ is a Demazure root corresponding to a primitive vector $v$ on an extremal ray $\rho$ of the cone $\sigma$.

The torus $\TT$ acts on a homogeneous function $\chi^m$ by multiplication by $t^m := t_1^{m_1} \ldots t_n^{m_n}$, where $m = (m_1, \ldots, m_n)$. Then for any $t \in \TT$, we obtain:
\begin{equation*}
	t \cdot \delta(\chi^m) =  t \cdot(\langle m, v \rangle \chi^{m+e}) = t^{m+e} \langle m, v \rangle \chi^{m+e},
\end{equation*}
\begin{equation*}
	\delta(t \cdot \chi^m) =  t^{m} \langle m, v \rangle \chi^{m+e}.
\end{equation*}

Therefore, the condition that an element of the torus $\TT$ belongs to the stabilizer $\Aut(B)_\delta$ of the homogeneous LND $\delta$ is expressed by equation $t^e = 1$. This equation defines an $(n-1)$-dimensional subtorus since the vector $e$ is primitive. We denote this subtorus by $\TT_\delta$. The subtorus commutes with $\delta$, so its action on $B$ can be restricted to $A := \Ker(\delta)$. Moreover, $A$ is a finitely generated subalgebra of $B$; thus, $\TT_\delta$ acts on $Y := \mathrm{Spec}(A)$. Hence, $\dim(Y) = n-1$ by Proposition~\ref{slice}.
\begin{lemma}\label{torlem}
	$Y$ is a toric variety and the action of the torus~$\TT_\delta$ is effective on $Y$.
\end{lemma}

\begin{proof}
	Suppose that the action of $\TT_\delta$ is not effective, then there exists $t\in \TT_\delta$ such that $t \cdot g = g$, $\forall g \in A$. Choose a $\TT$-homogeneous local slice $f$ for the derivation $\delta$ in the algebra $B$. Set $g := \delta(f)$. We will prove that $t \cdot f \neq f$.
	
	Suppose $t \cdot f = f$. The action of $t$ extends to $B_{g}$, and $t$ acts trivially on $A_{g} \subset B_{g}$. We have $ B_{g} = A_{g}[f]$ by Proposition~\ref{slice}, so $t$ acts trivially on $ B_{g}$, therefore it acts trivially on $B$. This contradicts the effectiveness of the $\TT$-action on $B$, thus $t \cdot f \neq f$.
	
	Since $t$ commutes with $\delta$ and since $g \in A$, we have:
	\begin{equation*}
		\delta (t \cdot f) = t \cdot \delta(f) = t \cdot g = g,
	\end{equation*}
	but $t \cdot f = t^m f \neq f$, where $m$ is the degree of $f$. Hence $t^m \neq 1$ and $\delta (t \cdot f) = t^m g \neq g$. This contradiction implies that the action of $\TT_\delta$ on $Y$ is effective.
	
	Thus, we have an effective action of the $(n-1)$-dimensional torus $\TT_\delta$ on $(n-1)$-dimensional variety $Y$; therefore, this action admits an open orbit, which completes the proof of the lemma.
\end{proof}

\begin{corollary}
	An element of the kernel $A$ is $\TT$-homogeneous if and only if it is $\TT_\delta$-homogeneous. 
\end{corollary}

\begin{proof}
	Since $\TT_\delta$ is a subtorus of $\TT$, $\TT$-homogeneity implies $\TT_\delta$-homogeneity in an obvious way. Now we prove the converse.
	
	Suppose that an element $a \in A$ is $\TT_\delta$-homogeneous. Consider the decomposition $a = \sum a_i$ into $\TT$-homogeneous components. Since $\delta(a) = \sum \delta(a_i) = 0$ and all $\delta(a_i)$ have distinct $\TT$-degrees, we obtain $a_i \in A \;\; \forall i$. Apply $t \in \TT_\delta$ to~$a$:
	\begin{equation*}
		\sum_i\lambda a_i = \lambda a = t \cdot a = \sum_i t \cdot a_i = \sum_i \lambda_i a_i.
	\end{equation*}
	This implies that all $a_i$ have the same $\TT_\delta$-degrees, then they are proportional, by Lemma~\ref{torlem}, hence $a$ is $\TT$-homogeneous.
\end{proof}

In what follows, when we say that an element of the kernel $A$ is homogeneous, we mean that it is homogeneous with respect to both tori $\TT$ and $\TT_\delta$.

From now on, we assume that $\delta$ is a maximal LND. A combinatorial criterion for an LND to be maximal in the case of normal $X$ was given in the previous section in Proposition~\ref{kriteriylnd}. For non-normal $X$, the condition of Proposition~\ref{kriteriylnd} remains sufficient.

Let $y_1, \ldots, y_k, z_1 \ldots, z_m$ be homogeneous, irreducible, pairwise non-proportional generators of the algebra $B$, such that $y_1, \ldots, y_k$ are generators of the algebra $A$. Denote the subgroup of $\Aut_\delta(B)$ consisting of automorphisms that send $\TT$-homogeneous elements of $B$ to $\TT$-homogeneous elements and permute the set $\{ y_1, \ldots, y_k \}$ by $S_\delta$.

\begin{theorem}\label{toric}
	Let $\delta$ be a maximal homogeneous LND on a toric variety $X$ with an effective action of a torus $\TT$. Then
	\begin{enumerate}
		\item[1)] The group $S_\delta$ is isomorphic to a subgroup of the symmetric group $S_k$ on the set $y_1, \ldots, y_k$. In particular, the group $S_\delta$ is finite.
		\item[2)] $\Aut(B)_\delta = (S_\delta \ltimes \TT_\delta) \ltimes \mathcal{U}(\delta).$
	\end{enumerate}
\end{theorem}
\begin{proof}
	Let $\phi \in \Aut(B)_\delta$. Every irreducible semi-invariant element of the algebra~$A$ is proportional to some $y_j$. Hence, we have $\phi(y_i) = \lambda_i y_{\tau(i)}$ for some $\lambda_i \in \KK^\times$ and a permutation $\tau \in S_k$ by Proposition~\ref{stab}. Thus, we obtain a group homomorphism:
	\begin{equation*}
		\Phi \colon \Aut_\delta(B) \rightarrow S_k, \;\;\; \phi \mapsto \tau. 
	\end{equation*}
	
	We can find an element $t \in \TT_\delta$ such that $t \circ \phi(y_i) = y_{\tau(i)}$ by Lemma~\ref{torlem}. Set $\psi := t \circ \phi$.
	
	Suppose that for $\psi \in \Aut(B)_\delta$ we have $\Phi(\psi) = \tau$. Choose a $\TT$-homogeneous local slice $f \in B$. Then $g := \delta(f)$ is homogeneous; therefore, the element $G := \psi (g)$ is homogeneous by Proposition~\ref{stab}. Decompose $\psi(f)$ into a sum of $\TT$-homogeneous elements: $\psi(f) = \sum_i F_i$.
	
	Since $\delta$ is a homogeneous LND, the decomposition of $\delta \circ \psi(f)$ into a sum of homogeneous elements has the form $\delta \circ \psi(f) = \sum_i \delta(F_i)$. Moreover, $\delta \circ \psi(f) = \psi(g) = G$ is a homogeneous element. This implies that all but one of the $F_i$ belong to the algebra $A$, that is, $\psi(f) = F + q$, where $q \in A$ and $F$ is a homogeneous local slice.
	
	Let $h \in B$ be $\TT$-homogeneous. Then, by Proposition~\ref{slice}, we obtain
	\begin{equation*}
		h = \frac{1}{g^l}  \sum_{i = 0}^N p_i f^i \;\;\text{or} \;\; hg^l = \sum_{i = 0}^N p_i f^i,
	\end{equation*}
	where $p_i \in A$. From the decomposition of the right-hand side into a sum of $\TT$-homogeneous elements we have $hg^l = f^r p$, where $l,r \in \ZZ_{\geqslant 0}$ and $p \in A$ is homogeneous. Consequently, $\psi(h)G^l = \nonumber \\ = (F+q)^r \psi(p)$, with $\psi(p)$ homogeneous. It follows that $(F+q)^r \psi(p)$ is divisible by $G^l$. Since the decomposition of $(F+q)^r \psi(p)$ into homogeneous components contains the term $F^r \psi(p)$, we conclude that $F^r \psi(p)$ is divisible by $G^l$.
	
	The isomorphism $\psi$ can be extended to an isomorphism
	\begin{equation*}
		\widetilde \psi \colon B_g = A_g[f] \,\, \widetilde{\rightarrow} \,\, B_G = A_G[F+q].
	\end{equation*}
	
	Define an isomorphism
	\begin{equation*}
		\widetilde \pi_\tau \colon B_g = A_g[f] \,\, \widetilde{\rightarrow} \,\, B_G = A_G[F] 
	\end{equation*} 
	as follows: $\widetilde \pi_\tau|_A = \widetilde \psi|_A$ and $\widetilde \pi_\tau(f) = F$. Thus we obtain
	\begin{equation*}
		\widetilde \pi_\tau(h) = \frac{\widetilde \pi_\tau(f)^r \widetilde \pi_\tau(p)}{\widetilde \pi_\tau(g)^l} = \frac{F^r \psi(p)}{G^l} \in B,
	\end{equation*}
	that is, $\widetilde \pi_\tau$ can be restricted to the algebra $B$. 
	
	Set $\pi_\tau := \widetilde \pi_\tau|_B$. It can be proved that $\widetilde \pi_\tau^{-1}$ restricts to $B$ in the same manner. It is also clear that its restriction to $B$ is the inverse homomorphism to $\pi_\tau$, so $\pi_\tau$ is an isomorphism.
	
	Thus, for each permutation $\tau \in \mathrm{Im}(\Phi)$, we have constructed an isomorphism $\pi_\tau \colon B \simeq B$. Applying $\pi_\tau$ to the equality $hg^l = f^r p$, one easily sees that $\pi_\tau$ sends $\TT$-homogeneous elements $h$ to $\TT$-homogeneous elements. We now prove that the set of all such $\pi_\tau$ is a group $S_\delta \simeq \mathrm{Im}(\Phi)$.
	
	We show that $\pi_\alpha \circ \pi_\beta = \pi_{\alpha\beta}$. It easily implies that the set of all $\pi_\tau$ is a group that is isomorphic to a subgroup of $S_k$. Indeed, the automorphism $\pi_\alpha \circ \pi_\beta \circ \pi^{-1}_{\alpha\beta}$ acts trivially on $A$. Hence, it belongs to $\mathcal{U}(\delta)$ by Proposition~\ref{Ker}. However, it sends $\TT$-homogeneous elements to $\TT$-homogeneous elements, which implies that $\pi_\alpha \circ \pi_\beta \circ \pi^{-1}_{\alpha\beta} = \mathrm{id}_B$. 
	
	Suppose that an automorphism $\xi \in S_\delta$ satisfies the equation $\Phi(\xi) = \tau$. Then the automorphism $\pi_\tau^{-1} \circ \xi$ acts trivially on the algebra $A$ and sends $\TT$-homogeneous elements of $B$ to $\TT$-homogeneous ones. Therefore, this automorphism coincides with $\mathrm{id}_B$. This shows that the set of all $\pi_\tau$ coincides with the group $S_\delta$, which proves part (1) of the theorem.
	
	Now, for an arbitrary $\psi$ with $\Phi(\psi) = \tau$, the automorphism $\pi_\tau^{-1} \circ \psi$ acts trivially on~$A$. Then we have $\pi_\tau^{-1} \circ \psi \in \mathcal{U}(\delta)$, by Proposition~\ref{Ker}. Consequently, $\pi_\tau^{-1} \circ t \circ \varphi\in \mathcal{U}(\delta)$. Thus, the groups $S_\delta$, $\TT_\delta$ and $\mathcal{U}(\delta)$ generate $\Aut(B)_\delta$. We obtain the desired decomposition of $\Aut(B)_\delta$, since the intersection of $S_\delta$, $\TT_\delta$ and $\mathcal{U}(\delta)$ is trivial and $S_\delta$ sends homogeneous elements to homogeneous ones.
\end{proof}

\begin{remark}
	An automorphism in $S_\delta$ is completely determined by a permutation of the set $\{ y_1, \ldots, y_k \}$, as it can be seen from the proof of Theorem~\ref{toric}. Moreover, for any permutation $\tau \in S_k$, one can construct a homomorphism $\pi_\tau \colon B \rightarrow \mathrm{Quot}(B)$ as in Theorem~\ref{toric}. This homomorphism is an automorphism of the algebra $B$ if and only if $\pi_\tau$ sends each $z_i$ to $\lambda_i z_j$ for some $\lambda_i \in \KK^\times$. This provides a way to enumerate all such permutations $\tau \in S_k$ and algorithmically compute the group~$S_\delta$. Hence, we can compute the group $\Aut_\delta(B)$ for any given toric variety $X$ and maximal homogeneous LND $\delta$.
\end{remark}

\begin{example}
	Consider the variety $X$ given by the equation $xy^2 = zw$ with the following action of the three-dimensional torus $\TT$:
	\begin{equation*}
		\left( \begin{array}{c}
			t_1 \\
			t_2 \\
			t_3
		\end{array} \right) 
		\cdot 
		\left( \begin{array}{c}
			x \\
			y \\
			z \\
			w
		\end{array} \right)
		= \left( \begin{array}{c}
			t_2x \\
			t_2^{-1}t_3y \\
			t_1^{-1}t_2^{-1}t_3^2z \\
			t_1w
		\end{array} \right)
	\end{equation*}
	
	This variety corresponds to the cone $\sigma$ generated by the following vectors: $a = (0,0,1), \,\, b = \nonumber \\ = (2,0,1), \,\, c = (0,1,1), \,\, d = (1,1,1)$. Consider the Demazure root $e = (1,2,-1)$. The LND $$\delta := \delta_e = xw\frac{\partial}{\partial y} + 2x^2y \frac{\partial}{\partial z}$$ is maximal by Proposition~\ref{kriteriylnd}. Moreover, $\Ker(\delta) = \KK[x,w]$.
	
	There are neither automorphisms that send $y$ to $\lambda y$ and $z$ to $\mu z$, nor any automorphisms that send $y$ to $\lambda z$ and $z$ to $\mu y$ with $\lambda, \mu \in \KK^{\times}$ among the automorphisms of $B :=\KK[X]$ that permute $x$ and $w$. Hence, we conclude that the group $S_\delta$ is trivial. Then, by Theorem~\ref{toric}, the automorphism group $\Aut(B)_\delta = \TT_\delta \ltimes \mathcal{U}(\delta)$, where $\TT_\delta = \{(t_1, t_2, t_1t_2^2)\} \subset \TT$ is a two‑dimensional torus.

\end{example}

\section{Trinomial hypersurfaces}\label{s8}

This section provides the reader with the necessary information about trinomial varieties. For more details, see, for example, paper \cite[Construction 1.1]{HW}.

\begin{definition}
	A variety $X$ is called {\it rigid} if there are no non-trivial $\mathbb{G}_a$-actions on it.
\end{definition}
In other words, there is no nonzero LND of the algebra of regular functions $\mathbb{K}[X]$. Rigid varieties are not incurious since, for example, the group of automorphisms of a rigid variety contains a unique maximal torus. This fact allows describing all automorphisms of rigid varieties, as was done in~\cite{AG}. 

Define a monomial $T_{i}^{l_{i}} := T_{i1}^{l_{i1}} \ldots T_{in_i}^{l_{in_i}}$, where $n_0 \geq 0, n_1, n_2 \geq 1,$ and $l_{ij}$ are positive integers. Moreover, if $n_0 = 0$, then the corresponding monomial is considered to be equal to one. Denote $n := n_0 + n_1 +n_2.$

\begin{definition}
	A {\it trinomial hypersurface} is an affine variety given by an equation of the following form in $n$-dimensional affine space: 
	
	\begin{equation*}
		T_{0}^{l_{0}} + T_{1}^{l_{1}} + T_{2}^{l_{2}} = 0,
	\end{equation*}
	where $T_{ij}$ are variables.
\end{definition}

In the case $n_0 = 0$, a trinomial hypersurface is called \textit{a trinomial hypersurface of type I}. Otherwise, it is called \textit{a trinomial hypersurface of type II}. 

\begin{theorem}\label{rigid}\cite[Theorem 2]{G}. 
	A trinomial hypersurface $$X = \mathbb{V}(T_{0}^{l_{0}} + T_{1}^{l_{1}} + T_{2}^{l_{2}})$$ is not rigid if and only if one of the following conditions holds:
	
	1) there exist $i \in \{0,1,2\}$ and $a \in \{1,2, \dots, n_i\}$ such that $l_{ia} = 1,$
	
	2) $n_0 \neq 0$ and there exist $i \neq j \in \{0,1,2\}$ and $a \in \{1,2,\dots,n_i\}, b \in \{1,2,\dots,n_j\}$ such that $l_{ia} = l_{jb} = 2$ and for all $u \in \{1,2,\dots,n_i\}, v \in \{1,2,\dots, n_j\}$ the numbers $l_{iu}$ and $l_{jv}$ are even.
\end{theorem}

Trinomial hypersurfaces are the simplest particular case of trinomial varieties, which are given by systems of equations of the following form: 

$$c_0 T_{0}^{l_{0}} + c_1T_{1}^{l_{1}} + c_2T_{2}^{l_{2}} = 0,$$ 
where $T_{ij}$ are variables, $c_i \in \mathbb{K} \setminus \{0\}$ satisfy certain conditions. An exact definition of trinomial varieties can be found, for example, in~\cite[Construction 1.1]{HW}.

Denote $\mathfrak{g} = T_{0}^{l_{0}} + T_{1}^{l_{1}} + T_{2}^{l_{2}},$ and let $X$ be a trinomial hypersurface, given by the equation $\mathfrak{g} = 0.$ It can be checked that the polynomial $\mathfrak{g}$ is irreducible, hence the algebra $R(\mathfrak{g}) := \KK[T_{ij}]/(\mathfrak{g})$ of regular functions on the hypersurface $X$ has no zero divisors. We call such algebras $R(\mathfrak{g})$ {\it trinomial}.

Following~\cite[Construction 1.1]{HW}, consider the finest grading on the trinomial algebra $R(\mathfrak{g})$ such that all generators $T_{ij}$ are homogeneous. This grading corresponds to an effective action of a quasitorus $\HH$ on the trinomial hypersurface $X$.

Consider a $2 \times n$-matrix L corresponding to the trinomial algebra $\mathfrak{g}$ as follows:
\[ L =
\begin{pmatrix}
	-l_{0} & l_{1} & 0\\
	-l_{0} & 0 & l_{2}\\
\end{pmatrix}.
\]
Let $L^{T}$ be the transpose of L. Denote by $K$ the factor group $K = \ZZ^n / \mathrm{Im}(L^T)$ and by $Q:\ZZ^n \rightarrow K$ the canonical projection. Let $e_{ij} \in \ZZ^n, i = 0,1,2, j = 1,\ldots n_i$, be the canonical basis vectors. Then define a $K$-grading on the algebra $\KK[T_{ij}]$ as follows:
\begin{equation}\label{deg}
	\mathrm{deg}\;T_{ij} = Q(e_{ij}).
\end{equation}

Note that the sums $\mu := l_{i1}Q(e_{i1})+\ldots +l_{in_i}Q(e_{in_i}) \in K$ are equal for $i = 0,1,2$ and $\mathfrak{g}$ is a homogeneous polynomial of degree $\mu$. Hence, the equalities~(\ref{deg}) define a $K$-grading on $R(\mathfrak{g}),$ corresponding to an effective action of a quasitorus $\HH = \mathrm{Hom}_{\ZZ}(K, \KK^{\times})$ on $R(\mathfrak{g})$. This action has a natural extension to an action on the trinomial hypersurface $X$.

From now on, we will consider hypersurfaces $\mathrm{X}$ of the type I only, that is, hypersurfaces given by equations of the form:

\begin{equation*}
	1 + T_{1}^{l_{1}} + T_{2}^{l_{2}} = 0. 
\end{equation*}

Let $\partial$ be a homogeneous LND with respect to the quasitorus action. The LND is defined on $\mathbb{K}[X]$. Then the following lemma follows directly from~\cite[Lemma~3.4]{G1}:

\begin{lemma}\label{odnlemma}
	There is an index $i \in \{1,\dots,n_1\}$, such that for any $j \neq i$ we have $\partial(T_{1j}) = 0$.
\end{lemma}

Let $A$ be a finitely generated abelian group and let $R$ be an arbitrary $A$-graded algebra. Denote the multiplicative semigroup of invertible elements of $R$ by $R^{\times}$. And denote the multiplicative semigroup of homogeneous elements of $R$ by $R^+$.

\begin{definition}
	A non-zero element $a \in R^+ \setminus  R^{\times}$ is called {\it homogeneous irreducible} if the condition $a = bc,$ $b, c \in R^+,$ implies that either $b$ or $c$ is invertible.
\end{definition}

\begin{definition}
	An $A$-graded algebra $R$ is said to be {\it factorially graded} if any its non-zero non-invertable homogeneous element may be expressed as a product of homogeneous irreducible elements, and such expression is unique up to association and renumbering.
\end{definition}

\begin{proposition}\label{hf}\cite[Proposition~2.6]{HW}
	An algebra of regular functions on a trinomial hypersurface of the type I is factorially graded.
\end{proposition}

\section{Commuting homogeneous LNDs on trinomial hypersurfaces}\label{razdel9}

In this section, we describe certain trinomial hypersurfaces of the type I, such that there are homogeneous with respect to the torus action maximal LNDs on these hypersurfaces.

\textbf{Case 1:} there is a unique variable in the monomial $T_{1}^{l_{1}}$.\\
Without loss of generality, considering the criterion for an arbitrary trinomial hypersurface to be rigid (Theorem~\ref{rigid}), it can be assumed that the hypersurface is given by the equation of the following form:

\begin{equation}\label{vyr2}
	x_1 \ldots x_k y_1^{a_1} \ldots y_m^{a_m} = z^n + 1, 
\end{equation}
where $k \geq 1, m \geq 0, n \geq 1$ and $a_i > 1$ for all $i \in \{1,\ldots, m\}.$ If $n = 1$, this equation gives an affine space, thus we will assume $n  > 1$.

It follows from Lemma~\ref{odnlemma} and conditions for the derivation to be locally nilpotent that every irreducible homogeneous LND for the considered hypersurface is given by the equations listed below, up to proportionality. A more accurate description of homogeneous LNDs on trinomial varieties can be found in~\cite{Ras}. For every $ 1 \leqslant i \leqslant k$:
$$ \partial_i(x_i) = nz^{n-1}, $$
$$\partial_i(z) =  y_1^{a_1} \ldots y_m^{a_m}\prod_{j \neq i}x_j,$$
$$\partial_i(x_j) = \partial_i(y_s) = 0, j \neq i, 1 \leqslant s \leqslant m.$$

It is easy to see that $[\partial_i, \partial_j] \neq 0,\ i \neq j,$ for example, since $\partial_i \partial_j(x_i) = 0$, but $\partial_j \partial_i(x_i) = \nonumber \\ = \partial_j(nz^{n-1}) \neq 0.$

Moreover, replicas $\partial_i$ commute neither with $\partial_j$, nor with replicas of $\partial_j$ if $j \neq i$:
$$h\partial_i \partial_j(x_i) = 0,$$ 
\begin{equation}\label{vyr}
	\partial_j h\partial_i(x_i) = \partial_j(hnz^{n-1}) = nz^{n-1}\partial_j(h) + n(n-1)hz^{n-2}y_1^{a_1} \ldots y_m^{a_m}\prod_{s \neq j}x_s,  
\end{equation}
where $h$ is a homogeneous element of the kernel of $\partial_i$. Hence, considering Proposition~\ref{slice}, which gives a relation for the transcendence degree of the kernel of a derivation, we conclude that $h$ is a homogeneous polynomial that does not contain $z, x_i$. Therefore, the expression (\ref{vyr}) is non-zero. 

In fact, if this expression is equal to zero, then $z\partial_j(h)$ is divisible by $h$, since $\mathbb{K}[X]$ is factorially graded by Proposition~\ref{hf}. However, $h$ does not depend on $z$, therefore, $\partial_j(h)$ is divisible by $h$. So, it follows that $\partial_j(h) = 0$ and $$n(n-1)hz^{n-2}y_1^{a_1} \ldots y_m^{a_m} \prod_{s \neq j}x_s = 0,$$ which is not true.

Thus, the expression (\ref{vyr}) is equal to zero for any admissible $h$.

$$h\partial_i g\partial_j(x_i) = 0,$$
\begin{equation}\label{vyr1}
	g\partial_j h\partial_i(x_i) = g\partial_j(hnz^{n-1}) = gnz^{n-1}\partial_j(h) + n(n-1)hgz^{n-2} y_1^{a_1} \ldots y_m^{a_m}\prod_{s \neq j}x_s,   
\end{equation}
where $g$ is a homogeneous element of the kernel of $\partial_j$, that is, $g$ is a homogeneous polynomial that does not contain $z, x_j$. Therefore, the expression (\ref{vyr1}) is non-zero for any admissible $h$, it can be checked in the same manner as it was done above.

So, the following lemma is true, according to the Proposition~\ref{odnor}:
\begin{lemma}\label{perm1}
	Every homogeneous LND on a trinomial hypersurface of the form (\ref{vyr2}) is maximal.
\end{lemma}

\textbf{Case 2:} there are not less than two variables in the monomial $T_{1}^{l_{1}}$.\\
Without loss of generality, considering the criterion for an arbitrary trinomial hypersurface to be rigid (Theorem~\ref{rigid}), it can be assumed that the hypersurface is given by the equation of the following form:
\begin{equation}\label{vyr3}
	x_1 \ldots x_k y_1^{a_1} \ldots y_m^{a_m} = z_1^{l_1} \ldots z_n^{l_n} + 1, 
\end{equation}
where $k \geq 1, m \geq 0, n > 1$ and $a_i > 1$ for all $i \in \{1,\ldots, m\}.$ 
We will deal with the case $l_j > 1$ for all $j \in \{1,\ldots, n\}.$

It follows from Lemma~\ref{odnlemma} and conditions for the derivation to be locally nilpotent that every irreducible homogeneous LND on the considered hypersurface is given by the equations listed below, up to proportionality. A more accurate description of homogeneous LNDs on trinomial varieties can be found in~\cite{Ras}. For every $ 1 \leqslant i \leqslant k,  1 \leqslant j \leqslant n$:
$$ \partial_{ij}(x_i) = l_jz_j^{l_j-1}\prod_{s \neq j}z_s^{l_s}, $$
$$\partial_{ij}(z_j) =  y_1^{a_1} \ldots y_m^{a_m}\prod_{s \neq i}x_s,$$
$$\partial_{ij}(x_r) = \partial_{ij}(y_s) = \partial_{ij}(z_p) = 0, r \neq i, 1 \leqslant s \leqslant m, p \neq  j.$$
It is easy to see that $[\partial_{ij}, \partial_{ir}] = 0$, which is enough to check for $x_i, z_j.$

Considering cases $k = 1$ and $k>1$, we calculate commutators in the same manner as in the case 1 above. We conclude that replicas of the form $h$$\partial_{ij}$ are maximal LNDs, where $h$ is a homogeneous element of the kernel of $\partial_{ij}$ and $h$ contains all $z_s, s \neq j$. That is, $h$ is a homogeneous polynomial that does not contain $z_j, x_i$ and contains all $z_s, s \neq j$. Moreover, if LND $h$$\partial_{ij},$ where $h$ is a homogeneous element of $\mathrm{Ker}(\partial_{ij}) = \KK[x_1,\ldots,\hat{x_i}, \ldots, x_k, y_1, \ldots, y_m, z_1, \ldots, \hat{z_j}, \ldots, z_n]$, is maximal, then the following holds:
\begin{equation*}
	[h\partial_{ij}, \partial_{ir}] = h[\partial_{ij}, \partial_{ir}] - \partial_{ir}(h) \partial_{ij} = -\partial_{ir} (h) \partial_{ij} \neq 0
\end{equation*}
for all $r \neq j.$ Hence, $h$ contains all $z_s, s \neq j$.

Therefore, using Proposition~\ref{odnor}, we obtain the following lemma:
\begin{lemma}\label{perm2}
	There are no maximal irreducible homogeneous LNDs on trinomial hypersurfaces of the form (\ref{vyr3}). However, their replicas of the form $h\partial_{ij}$, where $h$ is a homogeneous polynomial that does not contain $z_j, x_i$ and contains all $z_s, s \neq j,$ are maximal and only they are  maximal.
\end{lemma}

\begin{example}
	Consider a trinomial hypersurface of the form (\ref{vyr3}):
	
	$$xy^2 = z_1^2z_2^3 + 1.$$
	
	Irreducible homogeneous LNDs on this hypersurface are:
	
	$$\left\{
	\begin{gathered}
		\partial_1(x) = 2z_1z_2^3,\\
		\partial_1(z_1) = y^2,\\
		\partial_1(y) = \partial_1(z_2) = 0,\\
	\end{gathered}     
	\right.
	\qquad
	\left\{
	\begin{gathered}
		\partial_2(x) = 3z_1^2z_2^2,\\
		\partial_2(z_2) = y^2,\\
		\partial_2(y) = \partial_2(z_1) = 0.\\
	\end{gathered}
	\right.$$
	
	These LNDs commute, and they are not equivalent. So for every irreducible homogeneous LND on the considering hypersurface, there is an LND that is not equivalent to the first one but commutes with it. However, if we consider, for example, a replica $z_2\partial_1$, then we see that neither $\partial_2$ commutes with $z_2\partial_1$ nor do replicas of $\partial_2$, since for every homogeneous $g \in \Ker(\partial_2)$ we have: 
	$$z_2\partial_1g\partial_2(z_1) = 0,$$
	$$g\partial_2z_2\partial_1(z_1) = g\partial_2(z_2y^2) = gy^4 \neq 0.$$
	Thus, every homogeneous LND that commutes with $z_2\partial_1$ is equivalent to it.
	
\end{example}

\section{Isotropy subgroups of homogeneous LNDs on trinomial hypersurfaces}

This section describes isotropy subgroups of homogeneous LNDs on trinomial hypersurfaces of the form (\ref{vyr2}) and (\ref{vyr3}), which were considered in Section~\ref{razdel9}. 

Let $B:=\mathbb{K}[\mathrm{X}]$ denote the algebra of regular functions on a variety $X$, as per the above. There is an action of the automorphism group $\Aut(B)$ of the algebra $B$ on the set of all LNDs of $B$ by conjugation. Denote the stabilizer (or isotropy subgroup) of an LND $\delta$ under this action by $\Aut(B)_\delta := \{ \phi \in \Aut(B) \,|\, \phi \delta = \delta \phi \}$ as in Section 4. Let $\HH\subseteq \Aut(B)$ be the quasitorus that was introduced in Section~\ref{s8}.

Introduce several subgroups of the group $\Aut(B)_\delta$. Denote the intersection of $\HH$ and $\Aut(B)_\delta$ by $\HH_\delta$. The subgroup $\HH_\delta$ is a quasitorus. Let $\TT$ denote a connected component of $\HH$ and $\widehat{T}$ denote a connected component of $\HH_\delta$.

Consider an action of the symmetric group $\mathrm{S}_{k+m-1}$ on an affine space with coordinates $x_2,\ldots,x_k,y_1,\ldots, y_m$ by permutation of coordinates. Let $S_\delta$ denote the stabilizer of the monomial $h:= x_2 \ldots x_k y_1^{a_1} \ldots y_m^{a_m}$ under this action. It is clear that $S_\delta$ is a direct product of the symmetric groups that permute variables with the same degrees in the monomial. We are going to describe the group $\Aut(B)_\delta$.

\begin{remark}
	In Theorems~\ref{trinomial} and~\ref{trinomial1} we will consider $k \neq 1$. In case $k = 1$, a trinomial hypersurface of the form (\ref{vyr2}) is a special case of Danielewski variety. The isotropy groups $\Aut(B)_\delta$ for Danielewski varieties were calculated in~\cite{DL}. 
\end{remark}

\begin{theorem}\label{trinomial}
	Let $\delta$ be a homogeneous irreducible LND on a hypersurface of the form~(\ref{vyr2}). Then
	\begin{equation*}
		\Aut(B)_\delta = (S_\delta \ltimes \HH_\delta) \ltimes \mathcal{U}(\delta).
	\end{equation*}
\end{theorem}

\begin{proof}
	Every irreducible homogeneous LND for the considered hypersurface is given by the equations listed below, up to proportionality and renumbering of variables:
	$$ \delta(x_1) = nz^{n-1}, $$
	$$\delta(z) =  x_2 \ldots x_ky_1^{a_1} \ldots y_m^{a_m},$$
	$$\delta(x_j) = \delta(y_s) = 0, \text{when  } j \neq 1, 1 \leqslant s \leqslant m.$$
	
	Denote $A := \mathrm{Ker}(\delta)$. Let us prove that $A = \mathbb{K}[x_2, \ldots, x_k, y_1, \ldots, y_m]$. It is clear that $x_2, \ldots, x_k, y_1, \ldots, y_m \in A$. On the other hand, according to Proposition~\ref{slice}, $\mathrm{tr.deg.}A = \mathrm{tr.deg.}B - 1 = m + k - 1$. Hence, we see that the algebra $\mathbb{K}[x_2, \ldots, x_k, y_1, \ldots, y_m]$ coincides with $A$, since the first one is algebraically closed in $\mathbb{K}[B]$.
	
	Let us find the conditions under which an element $t \in \mathbb{T}$ commutes with $\delta$, where $t = (t_2, \ldots, t_k, s_1, \ldots, s_m)$:
	\begin{multline*}
		t \cdot (x_1,x_2, \ldots, x_k, y_1, \ldots, y_m, z) =\\ = (t_2^{-1} \ldots t_k^{-1}s_1^{-a_1} \ldots s_m^{-a_m}x_1, t_2x_2, \ldots, t_kx_k, s_1y_1, \ldots, s_my_m, z),
	\end{multline*}
	
	$$t \circ \delta(z) = t_2 \ldots t_ks_1^{a_1} \ldots s_m^{a_m} x_2 \ldots x_ky_1^{a_1} \ldots y_m^{a_m} = \delta \circ t(z) = x_2 \ldots x_ky_1^{a_1} \ldots y_m^{a_m},$$
	
	$$t \circ \delta(x_1) = nz^{n-1} = \delta \circ t(x_1) = t_2^{-1} \ldots t_k^{-1}s_1^{-a_1} \ldots s_m^{-a_m} n z^{n-1}.$$
	Therefore, the commute condition is the following: $t_2^{-1} \ldots t_k^{-1}s_1^{-a_1} \ldots s_m^{-a_m} = 1.$

	Consider $\phi \in \Aut(B)_\delta$. The functions $x_2, \ldots, x_k, y_1, \ldots, y_m$ are semi-invariant under the action of $\widehat{T}$. Moreover, their $\widehat{T}$-weights differ. It is easy to see that $\widehat{T}$ is a maximal torus in $\Aut(B)_\delta$. Every homogeneous irreducible LND on $B=\mathbb{K}[\mathrm{X}]$ for a hypersurface of the form~(\ref{vyr2}) is maximal by Lemma~\ref{perm1}. Hence, we can apply Proposition~\ref{stab}. Therefore, $\varphi$ permutes functions $x_i$, $2\leq i\leq k$, and $y_j$, $1\leq j\leq m$ and multiplies them by nonzero constants. 
	
	Note that the plinth ideal $\mathrm{pl}(\delta):=\delta(B) \cap A$ is principal ideal and is generated by the function $\delta(z)=h$. The ideal $\mathrm{pl}(\delta)$ is invariant under the action of $\Aut(B)_\delta$. Thus, $\varphi$ cannot permute variables that have different degrees in the monomial $h$. We have the following:
	
	$$\left\{
	\begin{gathered}
		\phi(y_1) = \lambda_1 y_{\sigma(1)},\\
		\phi(y_2) = \lambda_2 y_{\sigma(2)},\\
		\ldots\\
		\phi(y_m) = \lambda_m y_{\sigma(m)},\\
		\phi(x_2) = \mu_2 x_{\Delta(2)},\\
		\ldots\\
		\phi(x_k) = \mu_k x_{\Delta(k)},\\
	\end{gathered}
	\right.$$
	where $\lambda_i, \mu_j \in \mathbb{K}^{\times}, \sigma \in S_m, \Delta \in S_{k-1}$ are some permutations such that the equality $\sigma(i) = j$ holds if and only if $a_i= a_j.$
	
	It is easy to see that $S_\delta\subseteq \Aut(B)_\delta$. Consider an automorphism $\xi\in S_\delta$ that permutes $x_i$ by $\Delta$ and permutes $y_j$ by $\sigma$. Then $\varphi=\xi\circ\psi$, where $\psi(x_i)=\lambda_i x_i$ and $\psi(y_j)=\mu_j y_j$ for $2\leq i\leq k$ and $1\leq j\leq m$. 
	
	So, considering the composition of $\psi$ and an appropriate element $t \in \HH_\delta$, we obtain the automorphism $\zeta$, such that $\zeta(y_j)=y_j$ for $1\leq j\leq m$ and $\zeta(x_i)=x_i$ for $2\leq i\leq k-1$. Moreover, $\zeta(x_k) = \alpha x_{k}$. 
	Hence,
	$$\delta\zeta(z) = \zeta\delta(z) = \zeta(x_2 \ldots x_ky_1^{a_1} \ldots y_m^{a_m}) = \alpha x_2 \ldots x_k y_1^{a_1} \ldots y_m^{a_m}.$$
	
	It follows that:
	$$\zeta(z) = \alpha z + f, f \in \mathrm{Ker}(\delta).$$
	
	We have: $x_1=\frac{z^n+1}{h}$. Apply the automorphism $\zeta$ to each side of this equality:
	
	\begin{multline*}\zeta(x_1) = \zeta\left(\frac{z^n + 1}{h}\right) = \frac{(\alpha z + f)^n +1}{\alpha h} = \frac{\alpha^n z^n + 1 + (n\alpha^{n-1}z^{n-1}f + \ldots +f^n)}{\alpha h}=\\
		=\alpha^{n-1}x_1+\frac{ 1-\alpha^n + n\alpha^{n-1}z^{n-1}f + \ldots +f^n}{\alpha h}
		\in \mathbb{K}[\mathrm{X}].\end{multline*}

	Therefore, $f$ is divisible by $h$ in $A$ and $\alpha^n = 1$. Suppose $f = hg, g \in A.$ Then $\zeta(z) = \alpha z + hg$. Hence, applying $\exp(-g\delta) \in \mathcal{U}(\delta)$ to $\zeta(z)$, we obtain:  
	$$\mathrm{exp}(-g\delta) \circ \zeta(z) = \alpha z.$$
	$$\exp(-g\delta) \circ \zeta(x_1) = \frac{(\alpha z)^n +1}{\alpha h}=\frac{z^n +1}{\alpha h} = \frac{x_1}{\alpha}.$$
	
	So, the automorphism $\rho=\exp(-g\delta) \circ \zeta$ acts trivially on $x_2,\ldots, x_{k-1}$, $y_1,\ldots, y_m$, multiplies $z$ and $x_k$ by $\alpha\in\KK^\times$ and divides $x_1$ by $\alpha$. It is easy to see that $\rho\in \HH_\delta$.

	Thus, $\Aut(B)_\delta$ is generated by the subgroups $\mathcal{U}(\delta)$, $S_\delta$ and $\HH_\delta$. 
	
\end{proof}

\begin{example}
	Consider the following trinomial hypersurface of the form (\ref{vyr2}):
	$$x_1x_2y_1^2y_2^2y_3^7 = z^3 + 1.$$
	The LND
	$$\delta(x_1) = 3z^2,$$
	$$\delta(z) = x_2y_1^2y_2^7y_3^2,$$
	$$\delta(x_2) = \delta(y_1) = \delta(y_2) = \delta(y_3) = 0$$
	is maximal by Lemma~\ref{perm1}.
	Variables $x_1$ and $x_2$, $y_1$ and $y_2$ have the same degrees among variables of the monomial $h = x_1x_2y_1^2y_2^7y_3^2$. So in this case $S_\delta = S_2 \times S_2 \simeq \ZZ_2 \times \ZZ_2.$
	Consider the quasitorus $\HH = \TT \times K$, where its connected component $\TT$ is four-dimensional torus, $K$ is a finite group. The quasitorus $\HH$ acts on this variety. The torus $\TT$ acts on $x_1, x_2, y_1, y_2, y_3$ by the following rule: 
	\begin{equation*}
		\left( \begin{array}{c}
			t_1 \\
			t_2 \\
			t_3 \\
			t_4
		\end{array} \right) 
		\cdot 
		\left( \begin{array}{c}
			x_1 \\
			x_2 \\
			y_1 \\
			y_2 \\
			y_3 
		\end{array} \right)
		= \left( \begin{array}{c}
			t_1^{-1}t_2^{-2}t_3^{-2}t_4^{-7}x_1 \\
			t_1x_2\\
			t_2y_1 \\
			t_3y_2\\
			t_4y_3
		\end{array} \right)
	\end{equation*}
	and acts trivially on $z$. $K$ acts on $z$ by multiplication by cube roots of unity and acts trivially on $x_1, x_2, y_1, y_2, y_3$.
	
	Then $\hat{T}$ = $\{(t_2^{-2}t_3^{-2}t_4^{-7}, t_2, t_3, t_4)\} \subset \TT$ is a three-dimensional torus, the connected component of~$\HH_\delta.$
	
	Hence the group $\Aut(B)_\delta$ = $((\ZZ_2 \times \ZZ_2) \ltimes \HH_\delta) \ltimes \mathcal{U}(\delta)$ by Theorem~\ref{trinomial}.
\end{example}

Now consider the case of a hypersurface of the form (\ref{vyr3}).

Let $h\partial_{ij}$ be a homogeneous LND, where $h$ is a homogeneous polynomial that does not contain $z_j, x_i$ and contains all $z_s, s \neq j$. It was proved in Section 9 that for such a variety LND is maximal if and only if LND is of the form $h\partial_{ij}$. Hence, Proposition~\ref{stab} is true for these LNDs.

We may consider the following homogeneous LND without loss of generality (see Section 9):
$$ \delta(x_1) = hl_1z_1^{l_1-1}z_2^{l_2}\ldots z_n^{l_n}, $$
$$\delta(z_1) =  hy_1^{a_1} \ldots y_m^{a_m}x_2 \ldots x_k,$$
$$\delta(x_r) = \delta(y_s) = \delta(z_p) = 0, r \neq 1, 1 \leqslant s \leqslant m, p \neq  1,$$
where $h$ is a homogeneous polynomial that does not contain $z_1, x_1$ and contains all $z_s$, $s \neq 1$.

Consider an action of the symmetric group $\mathrm{S}_{k+m+n-2}$ on an affine space with coordinates $x_2,\ldots,x_k,y_1,\ldots, y_m, z_2,\ldots, z_n$ by permutation of coordinates. Let $S_\delta$ denote the intersection of the stabilizers of polynomials $h, h_2:=z_2^{l_2}\ldots z_n^{l_n}$ and $h_1:= x_2 \ldots x_k y_1^{a_1} \ldots y_m^{a_m}$ under this action. It is clear that $S_\delta$ is a direct product of the symmetric groups that permute variables $x_i$ among themselves, permute variables $y_j$ with the same degrees among themselves, permute variables $z_p$ with the same degrees among themselves in the monomials $h_1$ and $h_2$ and preserve the polynomial~$h$. We will describe $\Aut(B)_\delta$.

\begin{theorem}\label{trinomial1}
	Let $\delta$ be the described above LND on a hypersurface of the form (\ref{vyr3}). Then
	\begin{equation*}
		\Aut(B)_\delta = (S_\delta \ltimes \HH_\delta) \ltimes \mathcal{U}(\delta).
	\end{equation*}
\end{theorem}

\begin{proof}
	We prove that $\mathrm{Ker}(\delta)$ = $\mathbb{K}[x_2, \ldots, x_k, y_1, \ldots, y_m, z_2, \ldots, z_n].$ We will denote $A :=\mathrm{Ker}(\delta)$. 
	
	Then $x_2, \ldots, x_k, y_1, \ldots, y_m, z_2, \ldots, z_n \in A$. On the other hand, $\mathrm{tr.deg.}A = \mathrm{tr.deg.}B - 1 = m + k + n - 2$ by Proposition~\ref{slice}. It is known that the subalgebra $\mathbb{K}[x_2, \ldots, x_k, y_1, \ldots, y_m, z_2, \ldots, z_n]$ is algebraically closed in $\mathbb{K}[B]$, hence it coincides with $A$.
	
	Let $d_{i}$ be the least common multiple of numbers $l_1$ and $l_i$. Consider an element $t \in \mathbb{T}$, where $t = (t_2, \ldots, t_k, s_1, \ldots, s_m, r_2, \ldots, r_n).$ Let us find the conditions under which $t$ commutes with $\delta$:
	\begin{multline*}
		t \cdot (x_1,x_2, \ldots, x_k, y_1, \ldots, y_m, z_1, \ldots, z_n) =\\ = (t_2^{-1} \ldots t_k^{-1}s_1^{-a_1} \ldots s_m^{-a_m}x_1, t_2x_2, \ldots, t_kx_k, s_1y_1, \ldots, s_my_m, \\ r_2^{-\frac{d_2}{l_1}}\ldots r_n^{-\frac{d_n}{l_1}}z_1, r_2^{\frac{d_2}{l_2}}z_2 \ldots, r_n^{\frac{d_n}{l_n}}z_n),
	\end{multline*}
	
	\begin{multline*}
		t \cdot \delta(z_1) = t(h)t_2 \ldots t_ks_1^{a_1} \ldots s_m^{a_m} x_2 \ldots x_ky_1^{a_1} \ldots y_m^{a_m} = \delta \cdot t(z_1) =\\ = hr_2^{-\frac{d_2}{l_1}}\ldots r_n^{-\frac{d_n}{l_1}}x_2 \ldots x_ky_1^{a_1} \ldots y_m^{a_m},
	\end{multline*}
	
	\begin{multline*}
		t \cdot \delta(x_1) = t(h)r_2^{\frac{d_2}{l_1}}\ldots r_n^{\frac{d_n}{l_1}}l_1z_1^{l_1-1}z_2^{l_2}\ldots z_n^{l_n} = \delta \cdot t(x_1) =\\ = ht_2^{-1} \ldots t_k^{-1}s_1^{-a_1} \ldots s_m^{-a_m} l_1z_1^{l_1-1}z_2^{l_2}\ldots z_n^{l_n}.
	\end{multline*}
	
	Therefore, the commuting condition is the following: $$t(h)r_2^{\frac{d_2}{l_1}}\ldots r_n^{\frac{d_n}{l_1}}t_2 \ldots t_ks_1^{a_1} \ldots s_m^{a_m} = h.$$
	
	Consider $\phi \in \Aut(B)_\delta$. Functions $x_2, \ldots, x_k, y_1, \ldots, y_m, z_2, \ldots, z_n$ are semi-invariant under the action of $\widehat{T}$. Moreover, their $\widehat{T}$-wights differ. It is easy to see that $\widehat{T}$ is a maximal torus in  $\Aut(B)_\delta$. The LND of the considered form on $B=\mathbb{K}[\mathrm{X}]$ for hypersurface of the form (\ref{vyr3}) is maximal by Lemma~\ref{perm2}. Hence, we can apply Proposition~\ref{stab}. Therefore, $\varphi$ permutes functions $x_i$, $2\leq i\leq k$, $y_j$, $1\leq j\leq m$, and $z_p$, $2\leq p\leq n$, and multiplies them by nonzero constants. 
	
	Note that the plinth ideal $\mathrm{pl}(\delta)$ is principal ideal and is generated by the function $\delta(z_1)= hh_1$. The ideal $\mathrm{pl}(\delta)$ is invariant under the action of $\Aut(B)_\delta$.
	
	Then:
	$$\delta\phi(z_1) = \phi\delta(z_1) = \phi(hh_1) = \beta hh_1, \beta \in \mathbb{K}^{\times}.$$
	
	It follows that
	\begin{equation}\label{eq1}
		\phi(z_1) = \frac{\phi(hh_1)}{hh_1}z_1 + f = \beta z_1 + f, f \in \mathrm{Ker}(\delta).
	\end{equation}
	
	We have: $x_1=\frac{z_1^{l_1}\ldots z_n^{l_n}+1}{h_1}$. Apply the automorphism $\phi$ to each side of this equality:
	\begin{multline}\label{eq}\phi(x_1) = \phi\left(\frac{z_1^{l_1}\ldots z_n^{l_n} + 1}{h_1}\right) = \frac{(\beta z_1 + f)^{l_1}\phi (z_2^{l_2})\ldots \phi(z_n^{l_n}) +1}{\phi (h_1)} =\\ = \frac{\beta^{l_1} z_1^{l_1} \phi(z_2^{l_2})\ldots \phi(z_n^{l_n}) + 1 + (l_1\beta^{l_1-1}z_1^{l_1-1}f + \ldots +f^n)\phi(z_2^{l_2})\ldots \phi(z_n^{l_n})}{\phi(x_2 \cdots x_ky_1^{a_1} \cdots y_m^{a_m})} =\\ = \frac{\beta^{l_1} z_1^{l_1} \phi(z_2^{l_2})\ldots \phi(z_n^{l_n}) + x_1 \cdots x_ky_1^{a_1}\cdots y_m^{a_m} - z_1^{l_1}\cdots z_n^{l_n}}{\phi(x_2 \cdots x_ky_1^{a_1} \cdots y_m^{a_m})} + \\ + \frac{(l_1\beta^{l_1-1}z_1^{l_1-1}f + \ldots +f^n)\phi(z_2^{l_2})\ldots \phi(z_n^{l_n})}{\phi(x_2 \cdots x_ky_1^{a_1} \cdots y_m^{a_m})}.\end{multline}
	
	Therefore, since the condition $\phi(x_1) \in \mathbb{K [\mathrm{X}]}$ holds, the monomial $z_1^{l_1}\cdots z_n^{l_n}$ should either be divisible by $\phi(x_2 \cdots x_ky_1^{a_1} \cdots y_m^{a_m})$, or cancel out with the monomial $\beta^{l_1} z_1^{l_1} \phi(z_2^{l_2})\ldots \phi(z_n^{l_n})$. If the first case occurs, then $\phi$ permutes $z_j$ and $x_i$ or $z_j$ and $y_s$ for several $j, i, s$. It is easy to see that this leads to the contradiction with the condition $\phi(x_1) \in \mathbb{K [\mathrm{X}]}.$ The second case occurs if and only if $\phi$ does not permute $z_j$ and $x_i$ and does not permute $z_j$ and $y_s$ for all $i, j, s.$ However, the condition $\phi(x_1) \in \mathbb{K [\mathrm{X}]}$ implies the fact that $\phi$ does not permute $x_i$ and $y_j$ for all $i, j$ either. Moreover, by the same reasoning, it is easy to see that $\phi$ cannot permute variables that can be found in the equation of the hypersurface in different degrees. Hence, the monomials $h_1$ and $h_2$ are invariant under the action of $\phi$. Therefore, the polynomial $h$ is invariant under the action of~$\phi$ too.
	
	We have:
	$$\left\{
	\begin{gathered}
		\phi(y_1) = \lambda_1 y_{\sigma(1)},\\
		\phi(y_2) = \lambda_2 y_{\sigma(2)},\\
		\ldots\\
		\phi(y_m) = \lambda_m y_{\sigma(m)},\\
		\phi(x_2) = \mu_2 x_{\Delta(2)},\\
		\ldots\\
		\phi(x_k) = \mu_k x_{\Delta(k)},\\
		\phi(z_2) = \nu_2 z_{\pi(2)},\\
		\ldots\\
		\phi(z_n) = \nu_n z_{\pi(n)},\\
	\end{gathered}
	\right.$$
	where $\lambda_i, \mu_j, \nu_p \in \mathbb{K}^{\times}, \sigma \in S_m, \Delta \in S_{k-1}, \pi \in S_{n-1}$ are some permutations such that the equality $\sigma(i) = j$ holds if and only if $a_i= a_j$ and the equality $\pi(i) = j$ holds if and only if $l_i = l_j$, and in addition the following is true:
	\begin{multline*}
		h(x_2, \ldots, x_k, y_1, \ldots, y_m, z_2, \ldots, z_n) = \\ = h(x_{\Delta(2)}, \ldots, x_{\Delta(k)}, y_{\sigma(1)}, \ldots, y_{\sigma(m)}, z_{\pi(2)}, \ldots, z_{\pi(n)}). 
	\end{multline*}
	
	It is not difficult to see that $S_\delta\subseteq \Aut(B)_\delta$. Consider an automorphism $\xi\in S_\delta$ that permutes $x_i$ by $\Delta$, permutes $y_j$ by $\sigma$ and permutes $z_p$ by $\pi$. Then $\varphi=\xi\circ\psi$, where $\psi(x_i)=\lambda_i x_i$, $\psi(y_j)=\mu_j y_j$ and $\psi(z_p) = \nu_p z_p$ for $2\leq i\leq k, 1\leq j\leq m, 2\leq p \leq n$. 
	
	So, considering the composition of $\psi$ and an appropriate element $t \in \HH_\delta$, we obtain the automorphism $\zeta$ such that $\zeta(y_j)=y_j$ for $1\leq j\leq m, \zeta(x_i)=x_i$ for $2\leq i\leq k-1$ and $\zeta(z_p)=z_p$ for $2\leq p \leq n$. Moreover, $\zeta(x_k) = \alpha x_{k}$. 
	
	Hence, taking into account (\ref{eq1}) and (\ref{eq}), we have the following:
	\begin{multline*}\zeta(x_1) = \zeta\left(\frac{z_1^{l_1}\ldots z_n^{l_n} + 1}{h_1}\right) = \\ = \frac{\beta^{l_1} z_1^{l_1} z_2^{l_2}\ldots z_n^{l_n} + 1 + (l_1\beta^{l_1-1}z_1^{l_1-1}f + \ldots +f^n)z_2^{l_2}\ldots z_n^{l_n}}{\alpha h_1} = \\ \frac{\beta^{l_1}}{\alpha}x_1 + \frac{1 - \beta^{l_1} +(l_1\beta^{l_1-1}z_1^{l_1-1}f + \ldots +f^n)z_2^{l_2}\ldots z_n^{l_n}}{\alpha h_1}.\end{multline*}
	
	So, $f$ is divisible by $h_1$ in $A$ and $\beta^{l_1}= 1$. Suppose $f = h_1g, g \in A.$ Then $\zeta(z_1) = \beta z_1 + h_1g$. Hence, applying $\exp(-g\frac{\delta}{h}) \in \mathcal{U}(\delta)$ to $\zeta(z_1),$ we obtain:  
	$$\mathrm{exp}(-g\frac{\delta}{h}) \circ \zeta(z_1) = \beta z_1.$$
	$$\exp(-g\delta) \circ \zeta(x_1) = \frac{(\beta z_1)^{l_1}z_2^{l_2}\ldots z_n^{l_n} +1}{\alpha h_1}=\frac{z_1^{l_1}z_2^{l_2}\ldots z_n^{l_n} +1}{\alpha h_1} = \frac{x_1}{\alpha}.$$
	
	Therefore, we conclude that the automorphism $\rho=\exp(-g\frac{\delta}{h}) \circ \zeta$ acts trivially on $x_2,\ldots, x_{k-1}, y_1,\ldots, y_m, z_2, \ldots, z_n$, multiplies $z_1$ by $\beta$ and $x_k$ by $\alpha$, where $\alpha, \beta \in\KK^\times$, and divides $x_1$ by $\alpha$. It is easy to see that $\rho\in \HH_\delta$.
	
	Thus, $\Aut(B)_\delta$ is generated by the subgroups  $\mathcal{U}(\delta)$, $S_\delta$ and $\HH_\delta$. 
	
\end{proof}

\end{document}